\documentclass[10pt]{amsart}

\usepackage{amsmath}
\usepackage{amsthm}
\usepackage{amssymb}

\usepackage{diagrams}

\theoremstyle{plain}
\newtheorem{dfn}[subsection]{Definition}
\newtheorem{thm}[subsection]{Theorem}
\newtheorem{prp}[subsection]{Proposition}
\newtheorem{cor}[subsection]{Corollary}
\newtheorem{lma}[subsection]{Lemma}

\theoremstyle{remark}
\newtheorem{rmk}[subsection]{Remark}

\newtheorem{stit}[subsection]{}

\def\Vv{\mathcal{V}}
\def\Cat{\mathrm{Cat}}
\def\VCat{\Vv\mathrm{-}\Cat}
\def\V'Cat{\Vv'\mathrm{-}\Cat}
\def\VGrph{\Vv\mathrm{-Grph}}
\def\cat{\mathrm{cat}}
\def\Mon{\mathrm{Mon}}
\def\Seg{\mathrm{Seg}}

\def\Iso{\mathbb{I}}
\def\JJ{\mathbb{J}}
\def\wIso{\Iso_f}

\def\ito{\rightarrowtail}
\def\eqv{\overset\sim\lrto}

\def\HH{\mathbb{H}}
\def\AA{\mathbb{A}}
\def\BB{\mathbb{B}}

\def\KK{\mathbb{K}}
\def\LL{\mathbb{L}}
\def\Alg{\mathrm{Alg}}
\def\Mod{\mathrm{Mod}}
\def\lra{\leftrightarrows}

\def\Ob{\mathrm{Ob}}
\def\rto{\longrightarrow}
\def\lto{\longleftarrow}
\def\Pp{\mathcal{P}}
\def\Gg{\mathcal{G}}
\def\Ff{\mathcal{F}}
\def\Ho{\mathrm{Ho}}

\def\lrto{\longrightarrow}

\def\op{\mathrm{op}}

\def\GG{\mathbb{G}}

\def\CC{\mathbb{C}}

\def\onto{\twoheadrightarrow}

\def\eps{\epsilon}

\begin{document}
\title{On the homotopy theory of enriched categories}

\author{Clemens Berger and Ieke Moerdijk}

\date{January 4, 2013}


\begin{abstract}We give sufficient conditions for the existence of a Quillen model structure on small categories enriched in a given monoidal model category. This yields a unified treatment for the known model structures on simplicial, topological, dg- and spectral categories. Our proof is mainly based on a fundamental property of cofibrant enriched categories on two objects, stated below as the Interval Cofibrancy Theorem.\end{abstract}

\maketitle

\begin{flushright}\emph{Dedicated to the memory of Dan Quillen}\end{flushright}

\section*{Introduction}

Most categories arising naturally in mathematics are enriched in a symmetric monoidal category with more structure than the category of sets. In those cases where the enriching category comes equipped with an appropriate notion of homotopy, it is common to reformulate classical concepts of category theory in a homotopically meaningful way. From this point of view, the relevant notion of equivalence between enriched categories is that of a \emph{Dwyer-Kan equivalence} \cite{DK}, which was originally defined for categories enriched in simplicial sets, often just called simplicial categories. A map of simplicial categories is a Dwyer-Kan equivalence if it induces a weak homotopy equivalence on the simplicial hom-sets while on objects it is surjective ``up to homotopy equivalence''. In general, Dwyer-Kan equivalences do not have any kind of weak inverse, but they induce an equivalence of categories once the simplicial hom-sets are replaced with their sets of path-components. A similar notion of Dwyer-Kan equivalence exists for categories enriched in compactly generated spaces, in chain complexes or in symmetric spectra, to name only a few.

The theory of Quillen model categories \cite{Hi,Ho} provides a powerful framework to treat these examples in a systematic way. For instance, Bergner \cite{Be1} shows that the category of simplicial categories carries a Quillen model structure in which the weak equivalences are the Dwyer-Kan equivalences. Tabuada \cite{Ta,Ta2} proves a similar result for $dg$-categories as well as for categories enriched in symmetric spectra. These and other examples naturally lead to the question under which conditions a model structure on a symmetric monoidal category $\Vv$ induces a model structure on the category $\VCat$ of small categories enriched in $\Vv$.

Lurie \cite{Lu} proves a general result in this direction, which applies to categories $\Vv$ in which monomorphisms are cofibrations and some other conditions are satisfied (see Theorem \ref{Lurie} below). Such categories are in particular left proper. The aim of the present text is to prove an analogous result for symmetric monoidal categories which satisfy conditions complementary to Lurie's; in particular, they are supposed to be right proper. The reader will find a precise statement in Theorem \ref{canonical} below.

In those examples where both Lurie's and our conditions are satisfied, we show that the two model structures agree. In fact, most of the known examples of a model structure on $\VCat$ have a class of trivial fibrations and a class of fibrant objects which are directly definable in terms of the corresponding classes in $\Vv$. These two classes completely determine the model structure on $\VCat$, and we refer to model structures of this kind as \emph{canonical}.  The class of weak equivalences of the canonical model structure is thus uniquely determined, though not given in explicit terms. We prove that under our assumptions on $\Vv$, the weak equivalences of the canonical model structure are precisely the Dwyer-Kan equivalences. We actually deduce this result from the general fact that homotopy equivalences in $\Vv$-categories are \emph{coherent} whenever $\Vv$ satisfies our conditions. In the case of topologically enriched categories this is due to Boardman and Vogt \cite[Lemma 4.16]{BV}. In particular, Dwyer-Kan equivalences are now ``surjective up to coherent homotopy equivalence'', a property needed to characterise the fibrations of the canonical model structure by a right lifting property with respect to an explicit \emph{generating set} of trivial cofibrations. This generating set uses in an essential way the concept of \emph{$\Vv$-interval}, which is a special kind of $\Vv$-category on two objects. Much of the technical material in this article goes into the study of these $\Vv$-intervals.

The article is subdivided into three sections: Section \ref{s1} contains precise statements of our main results after a discussion of the necessary model-theoretical background; Section \ref{s2} proves the existence of a canonical model structure on $\VCat$ under certain conditions on the base category $\Vv$; Section \ref{s3} establishes the cofibrancy properties of $\Vv$-intervals needed for the existence of the canonical model structure.

\section{Definitions and main results}\label{s1}Let $\Vv$ be a cofibrantly generated monoidal model category (see \cite{Hi,Ho}). Structured objects in $\Vv$, such as monoids, modules for a fixed monoid, etc., often carry a Quillen model structure, which is \emph{transferred} from $\Vv$ in the sense that the fibrations and weak equivalences between these structured objects are detected by a forgetful functor to $\Vv$ (or a family of such in the multi-sorted case). These structured objects in $\Vv$ can in most cases be defined as algebras over a suitable non-symmetric coloured operad in \emph{sets}. This motivates the following definition:

\begin{dfn}\label{adequate}A monoidal model category $\Vv$ is called \emph{adequate} if\begin{itemize}\item the monoidal model structure is compactly generated;\item for any non-symmetric coloured operad $\,\Pp$ in sets, the category of $\,\Pp$-algebras in $\Vv$ carries a transferred model structure.\end{itemize}\end{dfn}

\noindent Mild conditions on $\Vv$ imply adequacy. To give a precise definition of our concept of compact generation, it is best to introduce the following terminology.

A class of maps in $\Vv$ is \emph{monoidally saturated} if it is closed in $\Vv$ under \emph{cobase change}, \emph{transfinite composition}, \emph{retract}, and under \emph{tensoring with arbitrary objects}. The \emph{monoidal saturation} of a class of maps $K$ is the least monoidally saturated class containing $K$. For brevity, let us call \emph{$\otimes$-cofibration} any map in the monoidal saturation of the class of cofibrations, and \emph{$\otimes$-small} (resp. \emph{$\otimes$-finite}) any object which is small (resp. finite) with respect to $\otimes$-cofibrations. The class of weak equivalences is \emph{$\otimes$-perfect} if it is closed under filtered colimits along $\otimes$-cofibrations.

\begin{dfn}\label{basics}A cofibrantly generated monoidal model category is \emph{compactly generated} if every object is $\otimes$-small, and the class of weak equivalences is $\otimes$-perfect.\end{dfn}

\noindent Any \emph{combinatorial} monoidal model category with a \emph{perfect} (i.e. filtered colimit closed) class of weak equivalences is compactly generated. Our definition of compact generation was chosen so as to include also the monoidal model category of compactly generated topological spaces whose objects are not small, but only $\otimes$-small, and whose class of weak equivalences is not perfect, but only $\otimes$-perfect.\footnote{A topological space is compactly generated (resp. weakly Hausdorff) if its compactly closed \cite[2.4.21(2)]{Ho} subsets are closed (resp. if its diagonal is compactly closed). The $\otimes$-perfectness of the class of weak equivalences holds for the monoidal model category of compactly generated spaces as well as for the monoidal model category of compactly generated weak Hausdorff spaces.

In the second (more familiar) case one uses that $\otimes$-cofibrations are closed $T_1$-inclusions and that compact spaces are finite with respect to closed $T_1$-inclusions, cf. Hovey \cite[2.4.1--5]{Ho}. In the first (more general) case one uses that $\otimes$-cofibrations are closed subspace inclusions $X\to Y$ with the additional property that each $y\in Y\backslash X$ belongs to a closed subset of $Y$ not intersecting $X$. Compact spaces are finite even with respect to the latter class, cf. Dugger-Isaksen \cite[A.3]{DI}.} In general, by Hovey's argument \cite[7.4.2]{Ho}, the existence of a generating set of cofibrations with \emph{finite} (resp. \emph{$\otimes$-finite}) \emph{domain and codomain} implies the perfectness (resp. $\otimes$-perfectness) of the class of weak equivalences. For us, the following corollary of $\otimes$-perfectness will play an important role (\emph{cf.} Section \ref{proof}):

\begin{lma}\label{transfinite}In a compactly generated monoidal model category the class of those weak equivalences which are $\otimes$-cofibrations is closed under transfinite composition.\end{lma}

\begin{prp}\label{adequate2}A compactly generated monoidal model category is \emph{adequate} if either of the following two conditions is satisfied:

\begin{enumerate}\item[(i)]$\Vv$ admits a monoidal fibrant replacement functor and contains a comonoidal interval object, cf. \cite{BM2};
\item[(ii)]$\Vv$ satisfies the monoid axiom of Schwede-Shipley, cf. Muro \cite{Mu}.\end{enumerate}\end{prp}

\noindent Recall that the \emph{monoid axiom} of Schwede-Shipley \cite{SS} requires the monoidal saturation of the class of trivial cofibrations to stay within the class of weak equivalences. If all objects of $\Vv$ are cofibrant, the monoid axiom is a consequence of the pushout-product axiom. In a compactly generated monoidal model category the monoid axiom can be rephrased in simpler terms (since the transfinite composition part has already been taken care of by Lemma \ref{transfinite}), namely \emph{tensoring a trivial cofibration with an arbitrary object yields a couniversal weak equivalence}, cf. \cite{BB}. Examples of adequate monoidal model categories include the category of simplicial sets equipped either with Quillen's or Joyal's model structure, the category of $dg$-modules equipped with the projective model structure, and the category of symmetric spectra with the levelwise or stable projective model structure. An example of an adequate, but non-combinatorial monoidal model category is the category of compactly generated topological spaces where both criteria \ref{adequate2}(i) and (ii) apply.\vspace{1ex}

For any set $S$, we denote by $\VCat_S$ the following category: the objects of $\VCat_S$ are $\Vv$-enriched categories with object-set $S$, and the morphisms of $\VCat_S$ are $\Vv$-functors which are the identity on objects. The following proposition was shown in \cite{BM2} (resp. \cite{SS2,Mu}) under the first (resp. second) hypothesis of the preceding proposition. Several other authors proved it for specific choices of $\Vv$. It is an obvious consequence of the definition of adequacy since $\VCat_S$ is the category of algebras for a non-symmetric $S\times S$-coloured set-operad, cf. \cite{BM2} and Section \ref{examples}e.

\begin{prp}\label{VCatS}For any adequate monoidal model category $\,\Vv$ and any set $S$, the category $\VCat_S$ admits a transferred model structure. This model structure is right (resp. left) proper if $\,\Vv$ is right proper (resp. all objects of $\,\Vv$ are cofibrant).\end{prp}

For any set-mapping $f:S\to T$, there is a Quillen pair$$f_!:\VCat_S\lrto\VCat_T:f^*$$the right adjoint of which is defined by $(f^*\BB)(x,y)=\BB(fx,fy)$ for $x,y\in S$.\vspace{1ex}

In this paper, we will address the problem when a suitable transferred model structure exists on the category $\VCat$ of \emph{all} small $\Vv$-enriched categories, obtained by letting $S$ vary over arbitrary (small) sets. In fact, the known examples suggest a more precise way of formulating this problem based on the following definitions. Recall that for any model structure the \emph{trivial} fibrations are the maps which are simultaneously fibrations and weak equivalences.

\begin{dfn}A $\Vv$-functor $f:\AA\to\BB$ between $\Vv$-categories is called a \emph{local weak equivalence} (resp. \emph{local fibration}) if for any objects $x,y\in\Ob(\AA)$, the induced map $$\AA(x,y)\rto\BB(fx,fy)$$ is a weak equivalence (resp. fibration) in $\Vv$. A $\Vv$-category is called \emph{locally fibrant} if the $\Vv$-functor to the terminal $\,\Vv$-category is a local fibration.

A model structure on $\VCat$ is called \emph{canonical} if its fibrant objects are the locally fibrant $\Vv$-categories and its trivial fibrations are the local trivial fibrations which are surjective on objects.\end{dfn}

\noindent Recall that a Quillen model structure is completely determined by its classes of trivial fibrations and of fibrant objects. Therefore, a canonical model structure on $\VCat$ is \emph{unique} when it exists, and hence we can speak of \emph{the} canonical model structure on $\VCat$. Our main problem can now be reformulated as follows: \begin{center}\emph{For which adequate monoidal model categories $\Vv$ does\\ the canonical model structure on $\VCat$ exist ?}\end{center}

\begin{rmk}\label{cofibration0}The \emph{cofibrations} of the canonical model structure can be characterised as those $\Vv$-functors $f:\AA\to\BB$ for which the set-mapping $f:\Ob(\AA)\to\Ob(\BB)$ on objects is \emph{injective}, and the induced $\Vv$-functor $f_!\AA\to\BB$ with fixed object set $\Ob(\BB)$ is a cofibration in $\VCat_{\Ob(\BB)}$. In particular, the inclusion $\VCat_S\to\VCat$ preserves cofibrations for any set $S$.\end{rmk}

\begin{stit}\label{known}The canonical model structure is known to exist in the following cases:\vspace{1ex}

(i) If $\Vv$ is the category of simplicial sets, then $\VCat$ is usually referred to as the category of simplicial categories. Bergner \cite{Be1} has shown that if $\Vv$ is equipped with the classical Quillen model structure, the canonical model structure on $\VCat$ exists. She gives explicit descriptions of the class of weak equivalences (the Dwyer-Kan equivalences \cite{DK}) and of generating sets of cofibrations and trivial cofibrations.

(ii) If $\Vv$ is the category of compactly generated topological spaces, then $\VCat$ is the category of topological categories. The existence of the canonical model structure on $\VCat$ can be proved by the same methods as in the previous example.

(iii) If $\Vv$ is the category of sets, equipped with the Quillen model structure in which the weak equivalences are the isomorphisms, then $\VCat$ is the category of small categories, and the canonical model structure is the one known as the naive, or folk model structure, see for instance Joyal-Tierney \cite{JT} or Rezk \cite{Re0}. The fibrations of this model structure are known as the so-called \emph{isofibrations}.

(iv) If $\Vv$ is the category of small categories, then $\VCat$ is the category of small $2$-categories. Lack \cite{La} has shown that if $\Vv$ is equipped with the model structure of (iii), then the canonical model structure on $\VCat$ exists. In fact, it is a monoidal model category under the Gray tensor product.

(v) Let $\Vv$ be the category of small $2$-categories with the Gray tensor product. In this case, $\Vv$-categories are a special kind of $3$-categories often referred to as Gray-categories (or semi-strict $3$-categories). Lack showed in \cite{La2} that again, if $\Vv$ is equipped with the model structure of (iv), then the canonical model structure on $\VCat$ exists. A suitable ``higher'' Gray tensor product on Gray-categories (which would allow one further iteration) is however not known.

(vi) Let $\Vv$ be the category of chain complexes of modules over a commutative ring $R$, equipped with the projective model structure. Tabuada \cite{Ta} has shown that the category $\VCat$ of $dg$-categories over $R$ admits a canonical model structure.

(vii) Let $\Vv$ be the category of symmetric spectra, equipped either with the levelwise projective or with the stable projective model structure. Tabuada \cite{Ta2} shows that also in these cases, $\VCat$ admits a canonical model structure.\end{stit}\vspace{1ex}

The following result of a more general nature is due to Lurie, see Proposition A.3.2.4 and Theorem A.3.2.24 in the Appendix A of \cite{Lu}. For a discussion of the notion of \emph{Dwyer-Kan equivalence} and of Lurie's \emph{invertibility axiom} \cite[A.3.2.12]{Lu}, we refer the reader to Definition \ref{coherence} and Remark \ref{Invertibility} below.

\begin{thm}\label{Lurie}Let $\,\Vv$ be a combinatorial monoidal model category such that\begin{enumerate}\item[(i)]the class of weak equivalences is closed under filtered colimits;\item[(ii)]every monomorphism is a cofibration;\item[(iii)]the invertibility axiom  holds.\end{enumerate}Then the canonical model structure on $\VCat$ exists and is left proper. The weak equivalences are precisely the Dwyer-Kan equivalences; the fibrations between fibrant objects are the local fibrations which induce an isofibration on path-components.\end{thm}

Note that by (i) $\Vv$ is compactly generated and by (ii) all objects of $\Vv$ are cofibrant so that $\Vv$ is \emph{adequate} and moreover \emph{left proper}. The main purpose of this paper is to establish the following result, which complements Lurie's result in some sense.

The notion of a \emph{generating set of $\,\Vv$-intervals} will be introduced in Definition \ref{interval}. Note that in many concrete cases the existence of a generating set of $\,\Vv$-intervals is automatic, see Lemma \ref{combinatorial} and Corollary \ref{allfibrant}.

\begin{thm}\label{canonical}Let $\,\Vv$ be an adequate monoidal model category such that

\begin{enumerate}\item[(i)]the monoidal unit is cofibrant;\item[(ii)]the underlying model structure is right proper;\item[(iii)]there exists a generating set of $\,\Vv$-intervals.\end{enumerate}Then the canonical model structure on $\VCat$ exists and is right proper. The weak equivalences are precisely the Dwyer-Kan equivalences; the fibrations are the local fibrations which have the path-lifting property with respect to $\Vv$-intervals.\end{thm}

\begin{proof}The existence of the canonical model structure is Theorem \ref{modelstructure}. The identification of the class of weak equivalences follows from Propositions \ref{Dwyer-Kan} and \ref{Boardman-Vogt}.\end{proof}

The category of simplicial sets fulfills the hypotheses of both theorems so that Bergner's result \cite{Be1} can be considered as a special instance of both theorems.\vspace{1ex}

Let $\Iso$ be the $\Vv$-category on $\{0,1\}$ representing a single isomorphism: thus, $\Iso(0,0)=\Iso(0,1)=\Iso(1,1)=\Iso(1,0)=I_\Vv$, the unit of $\Vv$.

\begin{dfn}\label{interval}A \emph{$\Vv$-interval} is a cofibrant object in the transferred model structure on $\VCat_{\{0,1\}}$, weakly equivalent to the $\Vv$-category $\,\Iso$.

A set $\,\Gg$ of $\Vv$-intervals is \emph{generating} if each $\Vv$-interval $\HH$ is retract of a trivial extension $\KK$ of a $\Vv$-interval $\GG$ in $\Gg$, i.e. if there exists a diagram in $\VCat_{\{0,1\}}$\begin{diagram}[small,heads=littlevee]\GG&\rEmbed^\sim_j&\KK&\pile{\rTo^r\\\lTo_i}&\HH\end{diagram}in which $\GG$ belongs to $\Gg$, $j$ is a trivial cofibration and $ri=id_\HH$.\end{dfn}

We emphasise that conditions (i) and (ii) in our theorem are essential, but (iii) is a relatively innocent condition of a set-theoretical nature. For instance:

 \begin{lma}\label{combinatorial}For every combinatorial monoidal model category $\Vv$ there exists a generating set of $\,\Vv$-intervals.\end{lma}

\begin{proof}Since $\Vv$ is combinatorial, the overcategory $\VCat_{\{0,1\}}/\Iso_f$ (where $\Iso_f$ denotes a fibrant replacement of $\Iso$) is combinatorial and hence has an accessible class of weak equivalences, cf. Rosicky \cite{Ro} and  Raptis \cite{Ra}. This implies that the class of cofibrant objects in $\VCat_{\{0,1\}}$ equipped with a weak equivalence to $\Iso_f$ is accessible, i.e. there exists a \emph{set} $\Gg$ of $\Vv$-intervals such that for any $\Vv$-interval $\HH$ there is an object $\GG$ in $\Gg$ and a map (necessarily a weak equivalence) $\GG\to\HH$. According to Brown's Lemma the latter factors as a trivial cofibration $j:\GG\to\KK$ followed by a retraction $r:\KK\to\HH$ of a trivial cofibration $i:\HH\to\KK$. This just expresses that $\Gg$ is a generating set of $\Vv$-intervals.\end{proof}

In concrete examples, it is often possible to describe a generating set of $\,\Vv$-intervals directly. If $\Vv$ is the category of simplicial sets, the class of $\Vv$-intervals with countably many simplices is generating, cf. Bergner \cite[Lemmas 4.2 and 4.3]{Be1}, and is essentially small. We also remark that if every object in $\Vv$ is fibrant (which is the case in examples \ref{known}(ii), (iii), (iv), (v) and (vi) above), any single $\Vv$-interval is already generating, cf. Lemma \ref{single} below. Since in the latter case $\Vv$ is also right proper, we obtain:

\begin{cor}\label{allfibrant}If $\,\Vv$ is an adequate monoidal model category with cofibrant unit, in which every object is fibrant, then the canonical model structure on $\VCat$ exists.\end{cor}

In those cases where Corollary \ref{allfibrant} applies, the fibrations of the canonical model structure can be characterised in a concise way, since the $W$-construction of \cite{BM1,BM2} provides an explicit generating $\Vv$-interval $W\Iso$ for $\VCat_{\{0,1\}}$. The latter represents \emph{coherent homotopy equivalences} (cf. Definition \ref{defequivalence}) so that the fibrations of the canonical model structure are those local fibrations which are path-lifting with respect to these coherent homotopy equivalences. This characterisation is known for the fibrations of examples \ref{known} (iii)-(v), cf. Lack \cite{La,La2}, but seems to be new for the fibrations of topologically enriched resp. $dg$- categories, cf. \ref{known}(ii), (vi).\vspace{1ex}

An adjunction between symmetric monoidal categories is called \emph{monoidal} if the left and right adjoints are symmetric monoidal functors, and if the unit and counit of the adjunction are monoidal transformations. A monoidal adjunction $\Vv\lra\Vv'$ induces a family of adjunctions $\VCat_S\lra\V'Cat_S$ (varying naturally in $S$) and therefore a ``global'' adjunction $\VCat\lra\V'Cat$. If $\Vv$ and $\Vv'$ are monoidal model categories with cofibrant unit and $\Vv\to\Vv'$ is a left Quillen functor which preserves the monoidal unit, then the induced functor $\VCat_{\{0,1\}}\to\V'Cat_{\{0,1\}}$ takes $\Vv$-intervals to $\Vv'$-intervals. Hence, the global right adjoint $\V'Cat\to\VCat$ preserves the (trivial) fibrations of the canonical model structures, and we get:

\begin{cor}Consider a monoidal Quillen adjunction $\Vv\lra\Vv'$ between monoidal model categories satisfying the hypotheses of Theorem \ref{canonical} and such that the left adjoint preserves the monoidal unit. Then the induced adjunction $\VCat\lra\V'Cat$ is again a Quillen adjunction with respect to the canonical model structures.\end{cor}

It is not difficult to check that the induced Quillen adjunction $\VCat\lra\V'Cat$ is a Quillen equivalence whenever the given Quillen adjunction $\Vv\lra\Vv'$ is. In particular, examples \ref{known}(i) and (ii) are related by a canonical Quillen equivalence.\vspace{1ex}

The proof of Theorem \ref{canonical} relies heavily on the following property of cofibrant $\Vv$-categories on two objects:

\begin{thm}[Interval Cofibrancy Theorem]\label{cofibrant}Let $\Vv$ be an adequate monoidal model category with cofibrant unit and let $\HH$ be a cofibrant $\Vv$-category on $\{0,1\}$. Then

\begin{enumerate}\item[(i)]The endomorphism-monoids $\HH(0,0)$ and $\HH(1,1)$ are cofibrant monoids;
\item[(ii)]$\HH(0,1)$ is cofibrant as a left $\HH(1,1)$-module and as a right $\HH(0,0)$-module.
\item[(iii)]$\HH(1,0)$ is cofibrant as a left $\HH(0,0)$-module and as a right $\HH(1,1)$-module.\end{enumerate}\end{thm}

The proof (or at least, our proof) of this theorem is technically involved, and will occupy the entire Section \ref{proofICT}. If $\Vv$ is the category of simplicial sets, part (i) goes back to Dwyer-Kan \cite{DK2} and has been used by Bergner \cite{Be1} in her proof of the canonical model structure on simplicially enriched categories.

Given two $\Vv$-intervals $\HH$ and $\KK$, one can amalgamate them by taking first the pushout in $\VCat$ given by identifying the object $1$ in $\HH$ with the object $0$ of $\KK$, and then restricting back to $\VCat_{\{0,1\}}$ where the new objects $0,1$ are the ``outer'' objects $0$ of $\HH$ and $1$ of $\KK$. The Interval Cofibrancy Theorem implies the following fact concerning the amalgamation of intervals, to be proved in Section \ref{proofIAL}.

\begin{lma}[Interval Amalgamation Lemma]\label{IAL}Let $\HH$ and $\KK$ be two $\Vv$-intervals. Then any cofibrant replacement (in $\VCat_{\{0,1\}}$) of their amalgamation $\HH*\KK$ is again a $\Vv$-interval.\end{lma}

\section{The model structure on $\Vv$-categories}\label{s2}

We establish in this section the existence of a canonical model structure on $\VCat$ provided $\Vv$ is a right proper, adequate monoidal model category with cofibrant unit and generating set of $\Vv$-intervals. Our proof uses the Interval Cofibrancy Theorem \ref{cofibrant} and the Interval Amalgamation Lemma \ref{IAL} which will be established in Section \ref{s3}. We also prove that the weak equivalences of the canonical model structure coincide with the \emph{Dwyer-Kan equivalences}. We first show in Proposition \ref{Dwyer-Kan} that this identification is quite obvious if a so-called \emph{coherence axiom} holds. We then show in Proposition \ref{Boardman-Vogt} that any adequate monoidal model category with cofibrant unit satisfies the coherence axiom. Our proof mimicks Boardman and Vogt's proof of the coherence axiom for enrichment in compactly generated topological spaces, cf. \cite[Lemma 4.16]{BV}. It is worthwhile noting that the coherence axiom is an immediate consequence of Lurie's \emph{invertibility axiom}, cf. Remark \ref{Invertibility}.\vspace{1ex}

Recall from the previous section that $I_\Vv$ denotes the unit of the monoidal model category $\Vv$, and that $\Iso$ denotes the $\Vv$-category on the object set $\{0,1\}$ such that $\Iso(0,0)=\Iso(0,1)=\Iso(1,0)=\Iso(1,1)=I_\Vv$ with composition maps given by the canonical isomorphism $I_\Vv\otimes_\Vv I_\Vv\cong I_\Vv$. Let $\wIso$ be a fibrant replacement of $\Iso$ in $\VCat_{\{0,1\}}$. Then, according to Definition \ref{interval}, a \emph{$\Vv$-interval} is a cofibrant $\Vv$-category $\HH$ on $\{0,1\}$ which comes equipped with a weak equivalence $\HH\eqv\wIso$ in $\VCat_{\{0,1\}}$. As usual, different choices of a fibrant replacement $\wIso$ of $\Iso$ lead to the same notion of $\Vv$-interval. Therefore, we can fix once and for all our preferred choice of $\wIso$. If the unit $I_\Vv$ of $\Vv$ is fibrant in $\Vv$, then $\Iso$ is fibrant in $\VCat_{\{0,1\}}$ so that we can put $\wIso=\Iso$.

\begin{lma}\label{single}If all objects of $\,\Vv$ are fibrant, then any single $\Vv$-interval is generating.\end{lma}

\begin{proof}Since all objects of $\,\Vv$ are fibrant, a $\Vv$-interval $\GG$ consists of a factorisation of the canonical inclusion $\{0,1\}\to\Iso$ into a cofibration $\{0,1\}\ito\GG$ followed by a weak equivalence $\GG\eqv\Iso$. We take any such $\GG$ as generating $\Vv$-interval.

We shall now realise an arbitrary $\Vv$-interval $\HH$ as retract of a trivial extension of $\GG$, cf. Definition \ref{interval}. Indeed, factor the weak equivalence $\GG\eqv\Iso$ into a trivial cofibration $\GG\overset{\sim}{\ito}\tilde{\GG}$ followed by a trivial fibration $\tilde{\GG}\overset{\sim}{\onto}\Iso$. Then, by cofibrancy of $\HH$ there is a lift $\HH\to\tilde{\GG}$. Factor this weak equivalence between cofibrant objects of $\Vv-\Cat_{\{0,1\}}$ (according to Brown's Lemma) into a trivial cofibration $j:\HH\to\KK$ followed by a retraction $\KK\to\tilde{\GG}$ of a trivial cofibration $\tilde{\GG}\overset{\sim}{\ito}\KK$. This yields the trivial extension $\GG\overset{\sim}{\ito}\tilde{\GG}\overset{\sim}{\ito}\KK$ while $j:\HH\to\KK$ admits a retraction $r:\KK\to\HH$, since $\HH$ is fibrant and $j$ a trivial cofibration.\end{proof}

\begin{stit}\textbf{Fibrations and weak equivalences in $\VCat$.}\label{fibration}\vspace{1ex}

\noindent A $\Vv$-functor $F:\AA\to\BB$ is said to be

\begin{itemize}
\item \emph{path-lifting} if it has the right lifting property with respect to $\{i\}\to\HH,\,i=0,1,$ for any $\Vv$-interval $\HH$;
\item \emph{essentially surjective} if for any object $b:\{1\}\to\BB$ there is an object $a:\{0\}\to\AA$ and a $\Vv$-interval $\HH$ together with a commutative diagram\begin{diagram}[small,UO]\{0\}&&\rTo^a&&\AA\\&\rdTo&&&\\&&\HH&&\dTo_F\\&\ruTo&&\rdTo&\\\{1\}&&\rTo_b&&\BB\end{diagram}in $\VCat$;\item \emph{a fibration} if it is a path-lifting local fibration;\item \emph{a weak equivalence} if it is an essentially surjective local weak equivalence.\end{itemize}\end{stit}

As usual, a \emph{trivial fibration} is defined to be a $\Vv$-functor which is both a fibration and a weak equivalence. A \emph{local trivial fibration} is defined to be a $\Vv$-functor which is both a local fibration and a local weak equivalence. A $\Vv$-category is \emph{(locally) fibrant} if the unique functor to the terminal $\Vv$-category is a (local) fibration.

\begin{lma}A locally fibrant $\Vv$-category is fibrant.\end{lma}

\begin{proof}We have to show that a local fibration with values in a terminal $\,\Vv$-category is automatically path-lifting; or, what amounts to the same, that any object-map $a:\{0\}\to\AA$ for a locally fibrant $\Vv$-category $\AA$ extends to any $\Vv$-interval $\HH$. It is obvious that $a$ extends to a $\Vv$-functor $\bar{a}:\Iso\to\AA$ such that $\bar{a}(0)=\bar{a}(1)=a(0)$. Since $\AA$ is fibrant in $\VCat_{\{0,1\}}$, $\bar{a}$ extends to a fibrant replacement $\Iso_f$ of $\Iso$. It suffices now to precompose this extension with the given weak equivalence $\HH\to\Iso_f$.\end{proof}

\begin{lma}\label{surjective}A $\Vv$-functor is a trivial fibration if and only if it is a local trivial fibration which is surjective on objects.\end{lma}

\begin{proof}The implication from left to right follows from the observation that a path-lifting and essentially surjective $\Vv$-functor is surjective on objects. For the implication from right to left we have to show that a local trivial fibration, which is surjective on objects, is essentially surjective and path-lifting.

The essential surjectivity follows by constructing a diagram like in Definition \ref{fibration} with $\HH$ replaced by $\Iso$, and precomposing it with a cofibrant replacement of $\Iso$ in $\VCat_{\{0,1\}}$. For the path-lifting property, given a map $b:\HH\to\BB$ and an object in $\AA$ over $b(0)$, we first use the surjectivity of the $\Vv$-functor $\AA\to\BB$ to also find an object over $b(1)$, and then use the left lifting property of the cofibration $\{0,1\}\to\HH$ with respect to $\AA\to\BB$ (cf. Remark \ref{cofibration0}) to obtain the required lift $\HH\to\AA$.\end{proof}

In view of the preceding two lemmas, the first part of Theorem \ref{canonical} can now be stated more explicitly as follows:

\begin{thm}\label{modelstructure}Let $\,\Vv$ be a right proper, adequate monoidal model category with cofibrant unit and a generating set of $\,\Vv$-intervals. Then $\VCat$ is a cofibrantly generated model category in which the weak equivalences are the essentially surjective local weak equivalences and the fibrations are the path-lifting local fibrations.\end{thm}

Before embarking on the proof in Section \ref{proof} below we establish some lemmas.

\begin{dfn}\label{defequivalence}Two objects $a_0,a_1$ of a $\Vv$-category $\,\AA$ are \emph{equivalent} if there exists a $\Vv$-interval $\,\HH$ and a $\Vv$-functor $\gamma:\HH\to\AA$ such that $\gamma(0)=a_0$ and $\gamma(1)=a_1$.

They are \emph{virtually equivalent} if they become $\Vv$-equivalent in some fibrant replacement $\AA_f$ of $\AA$ in $\VCat_{\Ob(\AA)}$.

They are \emph{homotopy equivalent} if there exist maps $\alpha:I_\Vv\to\AA_f(a_0,a_1)$ and $\beta:I_\Vv\to\AA_f(a_1,a_0)$ such that $\beta\alpha:I_\Vv\to\AA_f(a_0,a_0)$ (resp. $\alpha\beta:I_\Vv\to\AA_f(a_1,a_1)$) is homotopic to the arrow $I_\Vv\to\AA_f(a_0,a_0)$ (resp. $I_\Vv\to\AA_f(a_1,a_1)$) given by the identity of $\,a_0$ (resp. $\,a_1$).\end{dfn}

\begin{rmk}\label{pi0}Note that virtual (resp. homotopy) equivalence of objects in $\AA$ does not depend on the choice of the fibrant replacement $\AA_f$ of $\AA$ in $\VCat_{\Ob(\AA)}$. Note also that any $\Vv$-functor $\AA\to\BB$ takes (virtually, resp. homotopy) equivalent objects of $\AA$ to (virtually, resp. homotopy) equivalent objects in $\BB$.

Given a $\,\Vv$-category $\AA$, one can define an ordinary category $\pi_0(\AA)$ having the same objects as $\AA$, and with morphism-sets defined by$$\pi_0(\AA)(x,y)=\Ho(\Vv)(I_\Vv,\AA(x,y))=[I_\Vv,\AA_f(x,y)],$$(the latter identification with sets of homotopy classes uses the assumption that the unit of $\Vv$ is cofibrant). Then $x$ and $y$ are \emph{homotopy equivalent} in $\AA$ if and only if they become \emph{isomorphic} in $\pi_0(\AA)$.\end{rmk}

\begin{lma}\label{transitivity}For any $\Vv$-category $\AA$, equivalence and virtual equivalence are equivalence relations on the object set of $\AA$.\end{lma}

\begin{proof}Symmetry is obvious. For the reflexivity, observe that for any object $a_0$ of $\AA$ there is a canonical map $\Iso\to\AA$ witnessing that the identity of $a_0$ is an isomorphism; precomposing this map with a cofibrant replacement $\Iso_c\to\Iso$ in $\VCat_{\{0,1\}}$ yields the required self-equivalence of $a_0$. The non-trivial part of the proof concerns transitivity which follows from the Interval Amalgamation Lemma \ref{IAL}.\end{proof}

\begin{lma}\label{locallyfibrant}A local weak equivalence $F:\AA\to\BB$ reflects virtual equivalence of objects, i.e. if $Fa_0$ and $Fa_1$ are virtually equivalent in $\BB$, then $a_0$ and $a_1$ are virtually equivalent in $\AA$.\end{lma}

\begin{proof}Choose first a fibrant replacement $i_\BB:\BB\eqv\BB_f$ in $\VCat_{\Ob(\BB)}$. Next pull back $i_\BB$ along $F:\Ob(\AA)\to\Ob(\BB)$ to get the following diagram
\begin{diagram}[small]\AA&\rDashto^{i_\AA}&\AA_f\\\dTo^\alpha&&\dDashto_{\alpha'}\\F^*(\BB)&\rTo^{F^*(i_\BB)}&F^*(\BB_f)\\\dTo^\beta&&\dTo_{\beta'}\\\BB&\rTo_{i_\BB}&\BB_f\end{diagram}in which the broken arrows are defined by factoring $F^*(i_\BB)\alpha:\AA\to F^*(\BB_f)$ into a weak equivalence followed by a fibration in $\VCat_{\Ob(\AA)}$. By construction, $i_\AA,\alpha$ and $F^*(i_\BB)$ are local weak equivalences, hence so is $\alpha'$. Since $\beta$ and $\beta'$ induce isomorphisms on hom-objects, this implies that $\beta'\alpha':\AA_f\to\BB_f$ is a local trivial fibration. Therefore, any virtual equivalence $\gamma:\HH\to\BB_f$ between $Fa_0$ and $Fa_1$ can be lifted to a virtual equivalence $\tilde{\gamma}:\HH\to\AA_f$ between $a_0$ and $a_1$.\end{proof}

\begin{lma}\label{rightproper}If $\,\Vv$ is right proper then for any $\Vv$-category $\AA$, virtually equivalent objects of $\AA$ are equivalent.\end{lma}

\begin{proof}We can assume that $a_0,a_1$ are distinct objects of $\AA$, virtually equivalent through $\gamma:\HH\to\AA_f$ for some fibrant replacement $\AA_f$ of $\AA$ in $\VCat_{\Ob(\AA)}$. Pulling back $i_\AA:\AA\eqv\AA_f$ along the object set inclusion $a:\{a_0,a_1\}\to\Ob(\AA)$ we get the following diagram in $\VCat_{\{0,1\}}$:
\begin{diagram}[small]\LL_c&\rTo^\sim&\LL\SEpbk&\rOnto&a^*\AA\\&&\dTo^\sim&&\dTo_{a^*i_\AA}\\\HH&\rTo_\alpha^\sim&\KK&\rOnto_\beta&a^*\AA_f\end{diagram}in which $\beta\alpha$ is obtained by factoring $\gamma:\HH\to a^*\AA_f$ into a trivial cofibration followed by a fibration, $\LL$ is obtained by pullback, and $\LL_c$ is a cofibrant replacement of $\LL$. Since $\alpha:\HH\to\KK$ is a trivial cofibration, the weak equivalence $\HH\overset{\sim}{\to}\wIso$ extends to $\KK$; since $\Vv$ (and hence $\VCat_{\{0,1\}}$) is right proper, the vertical arrow $\LL\to\KK$ is a weak equivalence; therefore, $\LL_c$ is a $\Vv$-interval inducing the required equivalence between $a_0$ and $a_1$.\end{proof}

\begin{lma}\label{homotopyequivalence}In any $\Vv$-category $\AA$, virtually equivalent objects are homotopy equivalent.\end{lma}
\begin{proof}For any virtually equivalent objects $x,y$ of $\AA$ there exists a fibrant replacement $\AA_f$ of $\AA$ and a $\Vv$-interval $\HH$ together with a $\Vv$-functor $a:\HH\to\AA_f$ representing a path from $x$ to $y$ in $\AA_f$. By definition of a $\Vv$-interval, $\HH$ maps to a fibrant replacement $\wIso$ of $\Iso$ by a weak equivalence. Factor this weak equivalence into a trivial cofibration $\HH\overset{\sim}{\ito}\HH'$ followed by a trivial fibration $\HH'\overset{\sim}{\onto}\wIso$, and then extend $a$ to $a':\HH'\to\AA_f$ because $\AA_f$ is fibrant.

Next, consider the $\Vv$-category $\JJ$ on $\{0,1\}$ representing a single directed arrow, i.e. $\JJ(0,0)=\JJ(0,1)=\JJ(1,1)=I_\Vv$ (the monoidal unit of $\Vv$), but $\JJ(1,0)=\emptyset_\Vv$ (an initial object of $\Vv$) with evident composition law. The object-set inclusion $\{0,1\}\to\wIso$ factors then through $\JJ\to\Iso$ so that we get the following commutative diagram in $\VCat_{\{0,1\}}$\begin{diagram}[small,silent,UO]\{0,1\}&\rTo&\HH'&\rTo^{a'}&\AA_f\\\dTo&\ruDotsto&\dOnto_\sim\\\JJ&\rTo&\wIso\end{diagram}in which the lift $u:\JJ\to\HH'$ exists since $\{0,1\}\to\JJ$ is a cofibration in $\VCat_{\{0,1\}}$. We therefore obtain a $\Vv$-functor $a'u:\JJ\to\AA_f$, hence an arrow $\alpha:I_\Vv\to\AA_f(x,y)$. Interchanging the role of $0$ and $1$, we obtain an arrow $\beta:I_\Vv\to\AA_f(y,x)$. By construction, the composite arrow $\beta\alpha:I_\Vv\to\AA_f(x,x)$ (resp. $\alpha\beta:I_\Vv\to\AA_f(y,y)$) factors through $\HH'(0,0)$ (resp. $\HH'(1,1)$) and is thus homotopic to the arrow given by the identity of $x$ (resp. $y$).\end{proof}

\begin{lma}\label{homotopyequivalence2}Let $a_0,a_1$ (resp. $b_0,b_1$) be homotopy equivalent objects in a $\Vv$-category $\AA$. Then the hom-objects $\AA(a_0,b_0)$ and $\AA(a_1,b_1)$ are related by a zig-zag of weak equivalences in $\Vv$. Moreover, any $\Vv$-functor $F:\AA\to\BB$ induces a functorially related zig-zag of weak equivalences between $\BB(Fa_0,Fb_0)$ and $\BB(Fa_1,Fb_1)$.\end{lma}

\begin{proof}By definition, there exists a fibrant replacement $\AA_f$ of $\AA$, as well as arrows $\alpha:I_\Vv\to\AA_f(a_0,a_1)$ and $\beta:I_\Vv\to\AA_f(a_1,a_0)$ (resp. $\alpha':I_\Vv\to\AA_f(b_0,b_1)$ and $\beta':I_\Vv\to\AA_f(b_1,b_0)$) which are ``mutually homotopy inverse''. It then follows that $\beta^*(\alpha')_*:\AA_f(a_0,b_0)\to\AA_f(a_1,b_1)$ and $\alpha^*(\beta')_*:\AA_f(a_1,b_1)\to\AA_f(a_0,b_0)$ are mutually inverse isomorphisms in the homotopy category $\Ho(\Vv)$. The well-known saturation property of the class of weak equivalences of a Quillen model category implies then that $\beta^*(\alpha')_*$ and $\alpha^*(\beta')_*$ are weak equivalences in $\Vv$. The zig-zag of weak equivalences between $\AA(a_0,b_0)$ and $\AA(a_1,b_1)$ is obtained by concatenating with the weak equivalences $\AA(a_0,b_0)\eqv\AA_f(a_0,b_0)$ and $\AA(a_1,b_1)\eqv\AA_f(a_1,b_1)$.

Any $\Vv$-functor $F:\AA\to\BB$ extends to a commutative square\begin{diagram}[small]\AA&\rTo^\sim&\AA_f\\\dTo^F&&\dTo_{F_f}\\\BB&\rTo^\sim&\BB_f\end{diagram}in which $\AA_f$ (resp. $\BB_f$) is a fibrant replacement of $\AA$ (resp. $\BB$), and $F_f$ is a local fibration, cf. the proof of Lemma \ref{locallyfibrant}. Application of the $\Vv$-functor $F_f$ takes an arrow $I_\Vv\to\AA_f(x,y)$ to an arrow $I_\Vv\to\BB_f(Fx,Fy)$, and preserves ``connected components''. The existence of a functorially related zig-zag of weak equivalences between $\BB(Fa_0,Fb_0)$ and $\BB(Fa_1,Fb_1)$ then follows easily.\end{proof}

\begin{prp}\label{2outof3}If $\,\Vv$ is right proper, the class of weak equivalences of $\Vv$-categories satisfies the 2-out-of-3 property.\end{prp}

\begin{proof}Let $F:\AA\to\BB$ and $G:\BB\to\CC$ be $\Vv$-functors.

(i) Assume that $F$ and $G$ are weak equivalences. It is then immediate that $GF$ is a local weak equivalence; moreover $GF$ is essentially surjective by Lemma \ref{transitivity}, hence $GF$ is a weak equivalence.

(ii) Assume that $F$ and $GF$ are weak equivalences. It is then immediate that $G$ is essentially surjective. In order to prove that $G$ is a local weak equivalence, choose objects $b_0,b_1$ in $\BB$ and objects $a_0,a_1$ in $\AA$ such that $Fa_i$ is equivalent to $b_i$ for $i=0,1$. By Lemmas \ref{homotopyequivalence} and \ref{homotopyequivalence2} the hom-objects $\BB(F(a_0),F(a_1))$ and $\BB(b_0,b_1)$ are canonically weakly equivalent in $\Vv$, as are the hom-objects $\CC(GF(a_0),GF(a_1))$ and $\CC(G(b_0),G(b_1))$. We therefore get the following commutative diagram in $\Vv$:

\begin{diagram}[small,UO,silent]\AA(a_0,a_1)&\rTo^{F_{a_0,a_1}}&\BB(F(a_0),F(a_1))&\rLine^\sim&\BB(b_0,b_1)\\{(GF)_{a_0,a_1}}\!\!\!\!&\rdTo&\dTo_{G_{F(a_0),F(a_1)}}&&\dTo_{G_{b_1,b_2}}\\&&\CC(GF(a_0),GF(a_1))&\rLine^\sim&\CC(G(b_1),G(b_2))\end{diagram}
where the undirected horizontal lines stand for (functorial) zigzags of weak equivalences.
By assumption on $F$ and $GF$, $F_{a_0,a_1}$ and $(GF)_{a_0,a_1}$ are weak equivalences. Hence so are $G_{F(a_0),F(a_1)}$ and $G_{b_1,b_2}$ which shows that $G$ is a local weak equivalence.

(iii) Assume that $G$ and $GF$ are weak equivalences. It is then immediate that $F$ is a local weak equivalence. Since $\Vv$ is right proper, Lemmas \ref{locallyfibrant} and \ref{rightproper} imply that $G$ reflects equivalence of objects. It follows then from the essential surjectivity of $GF$ that $F$ is essentially surjective as well, and hence a weak equivalence.\end{proof}

\begin{rmk}It is unusual that the 2-out-of-3-property of the class of weak equivalence is not an immediate consequence of their definition. Readers who feel uncomfortable with this can use instead of Proposition \ref{2outof3} the Propositions \ref{Dwyer-Kan} and \ref{Boardman-Vogt} below, which show (independently of the existence of the model structure) that our weak equivalences coincide with the \emph{Dwyer-Kan equivalences} (cf. \ref{coherence}). The latter class is easily seen to fulfill the 2-out-of-3-property.

There is however one important point for those who wish to take Dwyer-Kan equivalences as weak equivalences from the very beginning. The innocent-looking Lemma \ref{surjective} relies on a compatible choice of the classes of weak equivalences and of fibrations. This was the raison d'\^etre for our definition of weak equivalences. If instead the Dwyer-Kan equivalences are chosen then, in order to validate Lemma \ref{surjective}, the fibrations should be defined as the local fibrations which induce an isofibration on path-components. The latter class is a priori different from our class of fibrations so that the existence of a generating set of trivial cofibrations for them is non-obvious, and most naturally achieved by an identification of the two classes of fibrations. This is the way all the authors of the cited examples \ref{known} proceed. The identification of the two classes of fibrations also follows from our coherence axiom \ref{CoherenceAxiom} and hence ultimately from Proposition \ref{Boardman-Vogt}.\end{rmk}
\vspace{1ex}

For the proof of Theorem \ref{modelstructure} we need a last lemma concerning the cobase change of free cofibrations of $\Vv$-categories. This technical lemma together with Lemma \ref{transfinite} will take care of ``transfinite compositions''. Recall from Section \ref{s1} that any map in the monoidal saturation of the class of cofibrations of $\Vv$ is called a $\otimes$-cofibration. A $\Vv$-functor $F:\AA\to\BB$ is called a \emph{local $\otimes$-cofibration} (resp. a \emph{free cofibration}) if for any objects $x,y$ in $\AA$, the induced map $\AA(x,y)\to\BB(Fx,Fy)$ is a $\otimes$-cofibration in $\Vv$ (resp. if $F$ is freely generated by a cofibration of $\Vv$-graphs, cf. Section \ref{examples}e).

\begin{lma}\label{cofibration}For any adequate monoidal model category $\Vv$, pushouts in $\VCat$ along a $\Vv$-functor $\phi:\AA\to\AA'$ which is injective on objects \begin{diagram}[small]\AA&\rTo^{\phi}&\AA'\\\dTo^{F}&&\dTo_{F'}\\\BB&\rTo_\psi&\NWpbk\BB'\end{diagram}have the following property: If $F$ is a free cofibration which is bijective on objects then $F'$ is a local $\otimes$-cofibration which is bijective on objects.\end{lma}

\begin{proof}The pushout decomposes into two pushouts by decomposing $\phi:\AA\to\AA'$ into a $\Vv$-functor $\AA\to \phi_!\AA$ (where $\phi$ also denotes the object mapping $\Ob\AA\to\Ob\BB$) followed by a $\Vv$-functor $\phi_!\AA\to\AA'$ which is bijective on objects:

\begin{diagram}[small]\AA&\rTo&\phi_!\AA&\rTo&\AA'\\\dTo^F&&\dTo_{F''}&&\dTo_{F'}\\\BB&\rTo\NWpbk&\psi_!\BB&\rTo\NWpbk&\BB'\end{diagram}
Since $F$ is a free cofibration which is bijective on objects, $F''$ as well is a free cofibration which is bijective on objects. Therefore, the right hand pushout can be considered as a pushout in $\Vv$-categories with fixed object set. As such, this pushout can be described as a sequential colimit in the category of $\Vv$-graphs with fixed object set. According to the Rezk-Schwede-Shipley formula for free extensions (\emph{cf.} \ref{examples}e) this sequential colimit takes the free cofibration $F''$ to a local $\otimes$-cofibration $F'$.\end{proof}

\subsection{Proof of Theorem \ref{modelstructure}}\label{proof}--\vspace{1ex}

We shall check the usual axioms CM1-CM5, where the cofibrations are defined by the left lifting property with respect to trivial fibrations.

By definition, the class of local trivial fibrations is characterised by the right lifting property with respect to$$I_{loc}=\{\JJ_{i,j}[X]\to\JJ_{i,j}[Y]\,|\,X\to Y\textrm{ a generating cofibration in }\Vv,\,i,j\in\{0,1\}\}$$ where the functor $\JJ_{0,1}[-]:\Vv\to\VCat_{\{0,1\}}$ associates to an object $X$ of $\Vv$ the $\Vv$-category on $\{0,1\}$ with $\JJ_{0,1}[X](0,0)=\JJ_{0,1}[X](1,1)=I_\Vv$ and $\JJ_{0,1}[X](0,1)=X$ and $\JJ_{0,1}[X](1,0)=\emptyset_\Vv$ with the canonical composition maps. For the definition of $\JJ_{i,j}[X]$ for general $i,j$, see Section \ref{examples}e. Therefore, Lemma \ref{surjective} implies that a generating set of cofibrations is given by adjoining to $I_{loc}$ the inclusion of the initial (empty) $\Vv$-category into the unit $\Vv$-category (having a single object with $I_\Vv$ as endomorphism monoid).

Similarly, the class of local fibrations is characterised by the right lifting property with respect to$$J_{loc}=\{\JJ_{i,j}[X]\to\JJ_{i,j}[Y]\,|\,X\to Y\textrm{ a generating trivial cofibration in }\Vv,\,i,j\in\{0,1\}\}.$$ Therefore, the definition of a fibration implies that a generating set of trivial cofibrations is given by adjoining to $J_{loc}$ the set of inclusions $\{0\}\to\GG$ where $\GG$ runs through a generating set $\Gg$ of $\Vv$-intervals.

Axiom CM1 concerning the existence of limits/colimits is clear; axiom CM2 about the class of weak equivalences is Proposition \ref{2outof3}. Axiom CM3 asks the classes of cofibrations, weak equivalences and fibrations to be closed under retracts. This holds for weak equivalences since essential surjectivity is closed under retracts. It holds for cofibrations and fibrations since both classes are definable by a lifting property. For the factorisation axioms CM4 we use Quillen's small object argument.

Observe first that it follows essentially from Lemma \ref{cofibration} and the explicit description of the generating cofibrations of $\VCat$ that their saturation (under cobase change and transfinite composition) belongs to the class of local $\otimes$-cofibrations. (Lemma \ref{cofibration} treats the case of an attachment which is injective on objects: the general case reduces to this one by means of the free monoid functor and Section \ref{examples}c). An adjunction argument and the $\otimes$-smallness of the objects in $\Vv$ then imply that those $\Vv$-categories, which are free on small $\Vv$-graphs, are small with respect to the saturation of the generating cofibrations of $\VCat$. Therefore Quillen's small object argument is indeed available and yields the existence of cofibration/trivial fibration factorisations. Observe also that since we required the class of $\otimes$-cofibrations in $\,\Vv$ to be closed under retract, each cofibration of $\Vv$-categories is a local $\otimes$-cofibration.

For the existence of trivial cofibration/fibration factorisations we furthermore have to show that the saturation of the set of generating trivial cofibrations is contained in the class of weak equivalences. Since the forgetful functor from $\Vv$-categories to $\Vv$-graphs preserves filtered colimits, Lemma \ref{transfinite} implies that local weak equivalences which are local $\otimes$-cofibrations are closed under transfinite composition. Moreover, essential surjectivity is also closed under transfinite composition. Therefore it suffices to show that the following two cobase changes in $\VCat$ yield $\Vv$-functors which are both local weak equivalences and local $\otimes$-cofibrations:

\begin{diagram}[small]\JJ_{i,j}[X]&\rTo&\AA&\quad\quad&\{0\}&\rTo&\AA\\\dTo^\sim&&\dTo&&\dTo^\sim&&\dTo\\\JJ_{i,j}[Y]&\rTo&\NWpbk\BB&\quad\quad&\GG&\rTo&\NWpbk\BB\end{diagram}
For the left hand cobase change this follows from Lemma \ref{cofibration} and from the existence of a transferred model structure on $\VCat_{\Ob(\AA)}$ because $\AA\to\BB$ can also be constructed as a pushout in $\VCat_{\Ob(\AA)}$. For the right hand cobase change we consider the following decomposition into two pushouts:
\begin{diagram}[small]\{0\}&\rTo&\AA\\\dTo^\phi&&\dTo_{\phi'}\\\GG_{0,0}&\rTo&\NWpbk\AA'\\\dTo^\psi&&\dTo_{\psi'}\\\GG&\rTo&\NWpbk\BB\end{diagram}
in which $\GG_{0,0}$ denotes a $\Vv$-category with a single object having $\GG(0,0)$ as endomor-phism-monoid. The $\Vv$-functor $\psi$ induces isomorphisms on hom-objects and is injective on objects; therefore (by the known purely algebraic properties of pushouts in $\VCat$), the $\Vv$-functor $\psi'$ also induces isomorphisms on hom-objects and is injective on objects, so certainly a local $\otimes$-cofibration. Since $\psi'$ is essentially surjective by construction, it is a local weak equivalence as well. It remains to be shown that $\phi'$ has the same properties. Since $\phi$ is bijective on objects, $\phi'$ can be constructed as a pushout in $\VCat_{\Ob(\AA)}$, via
\begin{diagram}[small]\{0\}&\rTo^x&\Ob(\AA)&\rTo&\AA\\\dTo^\phi&&\dTo&&\dTo_{\phi'}\\\GG_{0,0}&\rTo&\NWpbk x_!\GG_{0,0}&\rTo&\NWpbk\AA'\end{diagram}
 The Interval Cofibrancy Theorem \ref{cofibrant} implies that $\GG(0,0)$ is a weakly contractible, cofibrant monoid, so that the middle vertical arrow is a trivial cofibration in $\VCat_{\Ob(\AA)}$. It follows that $\phi'$ is a trivial cofibration in $\VCat_{\Ob(\AA)}$, and hence a local weak equivalence and a local $\otimes$-cofibration, as required.

Finally, the first half of lifting axiom CM5 follows from the definition of cofibrations. A well-known retract argument yields the second half of axiom CM5.\qed
\vspace{1ex}

The rest of this section is devoted to the identification of the weak equivalences of the canonical model structure with the so-called \emph{Dwyer-Kan equivalences}, often used in literature, see \cite{Be1,DK,La,La2,Ta,Ta2}. This identification establishes the second part of Theorem \ref{canonical}.

\begin{dfn}\label{coherence}A functor $\AA\to\BB$ between $\,\Vv$-categories is called a \emph{Dwyer-Kan equivalence} if it is a local weak equivalence with the property that the induced functor $\pi_0(f):\pi_0(\AA)\to\pi_0(\BB)$ is an equivalence of categories (cf. Remark \ref{pi0} for notation).\end{dfn}

Note that by Lemma \ref{homotopyequivalence}, each weak equivalence in the sense of Theorem \ref{modelstructure} is a Dwyer-Kan equivalence. The converse implication however is less obvious and amounts roughly to the property that any \emph{homotopy equivalence} is \emph{coherent} in the sense of Boardman and Vogt \cite{BV,Vogt}. This is a highly non-trivial property and probably one of the reasons for Lurie's \emph{invertibility axiom}, cf. Remark \ref{Invertibility}.

Recall that maps (resp. isomorphisms) in a $\Vv$-category are represented by $\Vv$-functors out of the category $\JJ$ (resp. $\Iso$) in $\VCat_{\{0,1\}}$ where $\JJ(i,j)=I_\Vv$ if $i\leq j$ and $\JJ(1,0)=\emptyset_\Vv$ (resp. $\Iso(i,j)=I_\Vv$ for all $i,j$). A \emph{cofibration} $\JJ\ito\HH$ into a $\Vv$-interval $\HH$ is called \emph{natural} if it fits into a commutative diagram of the form\begin{diagram}[small]\JJ&\rTo&\Iso\\\dTo&&\dTo_\sim\\\HH&\rTo^\sim&\Iso_f\end{diagram}where $\JJ\to\Iso$ is the obvious inclusion and $\Iso_f$ is a fibrant replacement of $\Iso$.

\begin{dfn}\label{CoherenceAxiom}A \emph{homotopy equivalence} between two objects of a $\Vv$-category $\AA$ is called \emph{coherent}, if the representing $\Vv$-functor $\alpha:\JJ\to\AA_f$ (cf. Lemma \ref{homotopyequivalence}) extends along a natural cofibration into a $\Vv$-interval $\HH$, as in\begin{diagram}[small,UO]\JJ&\rTo&\AA_f\\\dTo&\ruDotsto&\\\HH&&\end{diagram}A monoidal model category $\Vv$ is said to satisfy the \emph{coherence axiom} if every homotopy equivalence in any $\Vv$-category is coherent.\end{dfn}

\begin{rmk}\label{Invertibility}The \emph{invertibility axiom} of Lurie can be reformulated as follows (\emph{cf.} \cite[A.3.2.14]{Lu}): for any homotopy equivalence $\alpha:\JJ\to\AA_f$ and any natural cofibration $\JJ\ito\HH$ into a $\Vv$-interval $\HH$, the right vertical map in the pushout\begin{diagram}[small]\JJ&\rTo^\alpha&\AA_f\\\dTo&&\dTo\\\HH&\rTo&\AA_f\{\alpha^{-1}\}\end{diagram}is a weak equivalence. In other words: \emph{inverting a homotopy equivalence in a homotopy invariant way does not change the homotopy type}. Lurie's invertibility axiom in fact \emph{implies} our coherence axiom. Indeed, since $\JJ\ito\HH$ is a cofibration, its pushout $\AA_f\to\AA_f\{\alpha^{-1}\}$ is actually a trivial cofibration which has a retraction $\AA_f\{\alpha^{-1}\}\to\AA_f$ because $\AA_f$ is fibrant. The existence of the composite $\Vv$-functor $\HH\to\AA_f\{\alpha^{-1}\}\to\AA_f$ shows then that the homotopy equivalence $\alpha$ is coherent.\end{rmk}

\begin{prp}\label{Dwyer-Kan}Let $\Vv$ be a monoidal model category which is right proper and satisfies the coherence axiom. Then the class of essentially surjective local weak equivalences coincides with the class of Dwyer-Kan equivalences.\end{prp}

\begin{proof}Dwyer-Kan equivalences are local weak equivalences which on objects are surjective up to homotopy equivalence. Our notion of essential surjectivity means surjective up to equivalence. In general, equivalence implies virtual equivalence, and virtual equivalence implies homotopy equivalence (cf. Lemma \ref{homotopyequivalence}). If the coherence axiom holds then homotopy equivalence implies virtual equivalence; moreover, under right properness, virtual equivalence implies equivalence (cf. Lemma \ref{rightproper}). Therefore (under the coherence axiom \emph{and} right properness) the two notions of essential surjectivity coincide.\end{proof}

In \cite[Lemma 4.16]{BV}, Boardman and Vogt prove that homotopy equivalences in topological categories are coherent. For their proof they use a particular topological category on two objects, namely what we called elsewhere \cite{BM1,BM2} the \emph{Boardman-Vogt $W$-resolution} of the category $\Iso$ representing isomorphisms (throughout categories are considered as coloured non-symmetric operads with unary operations only). It was shown in \cite{BM1,BM2} that a general Boardman-Vogt $W$-resolution for $\Vv$-categories exists provided $\Vv$ possesses a suitable interval. We shall see in Lemma \ref{H} below that \emph{any} adequate monoidal model category $\Vv$ has such an interval $H$, so that the associated $W$-resolution $W(H,\Iso)$ of $\Iso$ is a $\Vv$-interval parametrising coherent homotopy equivalences in $\Vv$-categories. Boardman and Vogt's proof of the coherence axiom for topological categories now applies \emph{mutatis mutandis} to $\Vv$-categories. The following two lemmas of a general homotopical flavour are preparatory.

\begin{lma}[Vogt \cite{VogtHELP}]\label{weaklifting}A map $w:X\to Y$ between fibrant objects of a model category $\Vv$ is a weak equivalence if and only if for any cofibration between cofibrant objects $\gamma:A\to B$ and any commutative square of unbroken arrows\begin{diagram}[small,UO]A&\rTo&X\\\dTo^\gamma&\phi\,\,\,\ruDotsto\,\,\simeq&\dTo_w\\B&\rTo&Y\end{diagram}there exists a diagonal filler $\phi:B\to X$ which makes the upper triangle commute and the lower triangle commute up to homotopy.\end{lma}

\begin{proof}Assume first that $w$ is a weak equivalence. According to Brown's Lemma, $w$ factors then as a section $i:X\to X'$ of a trivial fibration $r:X'\to X$ followed by a trivial fibration $p:X'\to Y$. Since $\gamma$ is a cofibration, there is a filler $\psi:B\to X'$. Composing the latter with $r:X'\to X$ yields the required filler $\phi:B\to X$ making the upper triangle commute. By definition, we have $w\phi=w r \psi=pir\psi$, and it remains to be shown that this map is homotopic to $p\psi$. It suffices thus to show that $\psi$ and $ir\psi$ are homotopic. This holds since both maps get equal when composed with $r$, and composition with a \emph{trivial fibration} induces an injection $[B,r]:[B,X']\to[B,X]$ on homotopy classes.

Assume conversely that $w$ has the aforementioned lifting property and choose $A\to X$ (resp. $B\to Y$) to be a cofibrant replacement of $X$ (resp. $Y$). Passing to the homotopy category $\Ho(\Vv)$ shows then that the homotopy class $[\phi]\in[B,X]$ is both surjective and injective, hence bijective. Therefore, the homotopy class $[w]\in[X,Y]$ is bijective as well, so that $w:X\to Y$ is a weak equivalence.\end{proof}

\begin{lma}\label{parallellifting}Consider the following commutative diagram in $\Vv$:
\begin{diagram}[small,UO]&&X&&\\j\!\!\!\!&\ruTo&&\rdTo&\!\!\!\!w\\A&\rTo&A'&\rTo&Y\\\dTo^\gamma&&\dTo_{\gamma'}&&\\B&\rTo_\delta&B'&&\end{diagram}
in which $\gamma$ is a cofibration between cofibrant objects, $\gamma'$ is a trivial cofibration between cofibrant objects, and $w$ is a weak equivalence between fibrant objects. We assume moreover that the induced map $k:B\cup_AA'\to B'$ is a cofibration too.

Then there exists a pair of liftings $(\Phi:B\to X,\Psi:B'\to Y)$ which make the whole diagram commute.\end{lma}

\begin{proof}Since $\gamma'$ is a trivial cofibration and $Y$ is fibrant, there exists a lift $\tilde{\Psi}:B'\to Y$ making the diagram commute. Precomposing $\tilde{\Psi}$ with $\delta$ and invoking Lemma \ref{weaklifting} yields a diagonal filler $\Phi:B\to X$ such that $\Phi\gamma=j$ and such that $w\Phi$ and $\tilde{\Psi}\delta$ are homotopic. The universal property of pushouts yields a map $B\cup_AA'\to Y$ which is homotopic to $\tilde{\Psi}k$. The homotopy extension property of the cofibration $k$ permits to replace $\tilde{\Psi}$ with a lifting $\Psi:B'\to Y$ such that the composite map $\Psi k$ coincides with the given map $B\cup_AA'\to Y$ whence $\Psi\delta=w\Phi$ as required.\end{proof}

We are now ready to deduce the coherence axiom \ref{CoherenceAxiom} from the existence of a suitable $W$-resolution for $\Vv$-categories. Recall from \cite[Definition 4.1]{BM1} that an \emph{interval} $H$ for a \emph{monoidal model category} $\Vv$ with \emph{cofibrant unit} $I_\Vv$ consists in a factorisation of the folding map $I_\Vv\sqcup I_\Vv\lrto I_\Vv$ into a cofibration  followed by a weak equivalence $$I_\Vv\sqcup I_\Vv\overset{(0,1)}{\ito}H\eqv I_\Vv,$$ together with a monoid structure $H\otimes H\to H$ for which $0:I_\Vv\to H$ is neutral, and $1:I_\Vv\to H$ is absorbing.

\begin{lma}\label{H}Any adequate monoidal model category $\Vv$ with cofibrant unit $I_\Vv$ has an interval.\end{lma}
\begin{proof}Consider the following two adjoint pairs$$\Vv\overset{U_2}{\lra}\Mon_\Vv\overset{U_1}{\lra}\Seg_\Vv$$where $\Mon_\Vv$ denotes the category of monoids in $\Vv$, and $\Seg_\Vv$ the category of ``segments'' (i.e. monoids with absorbing element) in $\Vv$. The right adjoint functors $U_1,U_2$ are the obvious forgetful functors. By adequacy of $\Vv$ there are transferred model structures on monoids, resp. segments so that both adjoint pairs become Quillen pairs. Consider now the folding map $I_\Vv\sqcup I_\Vv\to I_\Vv$ as a map of segments and factor it as a cofibration $I_\Vv\sqcup I_\Vv\ito H$ followed by a weak equivalence $H\to I_\Vv$ in the transferred model structure on segments.

The segment $H$ would be an interval in $\Vv$ if the composite forgetful functor $U_2U_1$ took the cofibration of segments $I_\Vv\sqcup I_\Vv\ito H$ to a cofibration in $\Vv$. Since $I_\Vv\sqcup I_\Vv$ is a cofibrant segment it will be sufficient to show that $U_2U_1$ preserves cofibrations between cofibrant objects. For $U_2$ this follows from the discussion in Section \ref{examples}c. For $U_1$ observe that $U_1$ preserves pushouts and transfinite compositions, and its left adjoint consists in adjoining an external absorbing element. Thus the forgetful functor $U_1$ has the required property because the unit $I_\Vv$ is assumed cofibrant.\end{proof}

We showed in \cite[Theorem 5.1]{BM1} that for any $\Sigma$-cofibrant symmetric operad $P$ in a monoidal model category $\Vv$ with cofibrant unit $I_\Vv$ and interval $H$ there exists a canonical cofibrant replacement $W(H,P)\eqv P$ in the category of symmetric operads. As mentioned above, any $\Vv$-category $\AA$ can be considered as a coloured non-symmetric operad in $\Vv$ with unary operations only. This point of view is discussed in more detail in \cite{BM2}. In particular, the same proof as for \cite[5.1]{BM1} yields for any \emph{well-pointed} $\Vv$-category $\AA$ in a monoidal model category $\Vv$ with cofibrant unit $I_\Vv$ and interval $H$ a canonical cofibrant replacement $W(H,\AA)\eqv\AA$ in the category of $\Vv$-categories. Here, the term \emph{well-pointed} means that the reflexive $\Vv$-graph underlying $\AA$ is cofibrant in the category of reflexive $\Vv$-graphs, cf. Section \ref{examples}e. Indeed, the analog of \cite[5.1]{BM1} for $\Vv$-categories states in more precise terms that the counit $\Ff_*(\AA)\to\AA$ of the free-forgetful adjunction between reflexive $\Vv$-graphs and $\Vv$-categories (with fixed object-set) factors as a cofibration $\Ff_*(\AA)\ito W(H,\AA)$ followed by a weak equivalence $W(H,\AA)\eqv\AA$.

For our purpose only the special case $\AA=\Iso$ is relevant. Observe that the reflexive $\Vv$-graph underlying $\Iso$ is indeed cofibrant since the unit $I_\Vv$ is cofibrant. In particular, the $\Vv$-category $W(H,\Iso)$ is a $\Vv$-interval in the sense of Definition \ref{interval}. Moreover, the inclusion $\JJ\to\Iso$ induces a cofibration of the underlying reflexive $\Vv$-graphs. Therefore, we get cofibrations of $\Vv$-categories $$\JJ=\Ff_*(\JJ)\ito\Ff_*(\Iso)\ito W(H,\Iso)$$from which it follows that the composite map $\JJ\ito W(H,\Iso)$ is a \emph{natural cofibration} in the sense used for the formulation of the coherence axiom \ref{CoherenceAxiom}.

\begin{prp}[cf. Lemma 4.16 of Boardman-Vogt \cite{BV}]\label{Boardman-Vogt}Any adequate monoidal model category $\Vv$ with cofibrant unit $I_\Vv$ satisfies the coherence axiom.\end{prp}

\begin{proof}Since by Lemma \ref{H}, $\Vv$ possesses an interval $H$, it will be sufficient to show that any homotopy equivalence $\alpha:\JJ\to\AA_f$ extends along the natural cofibration $\JJ\ito W(H,\Iso)$. For a given interval $H$, the $W$-resolution $W(H,\Iso)$ of the $\Vv$-category $\Iso$ is constructed as a sequential colimit of reflexive $\Vv$-graphs on the vertex-set $\{0,1\}$ $$\cdots\ito W_k(H,\Iso)\ito W_{k+1}(H,\Iso)\ito\cdots$$where for each $k\geq 0$, the reflexive $\Vv$-graph $W_k(H,\Iso)$ is obtained from the reflexive $\Vv$-graph $W_{k-1}(H,\Iso)$ by attachment of two $k$-cubes $H^{\otimes k}$, one for each alternating string of length $k+1$,\begin{diagram}[small]\eps_0&\rTo^{\phi_0}&\eps_1&\rTo^{\phi_1}&\eps_2&\rTo&\cdots&\rTo&\eps_k&\rTo^{\phi_k}&\eps_{k+1}\end{diagram}where $\eps_i\in\{0,1\}$ and $\eps_i\not=\eps_{i+1}$ and $\phi_i$ stands for the ``single'' element of either $\Iso(0,1)$ or $\Iso(1,0)$ according to the value of $(\eps_i,\eps_{i+1})$. The vertex $\eps_0$ (resp. $\eps_{k+1}$) is the domain (resp. codomain) of the attached $k$-cube. The $k$-cube $H^{\otimes k}$ itself can be considered as a family of ``waiting times'' at the $k$ inner vertices of the corresponding string. The $\Vv$-category structure on $W(H,\Iso)$ is induced by concatenation of strings where waiting time $1:I_\Vv\to H$ is assigned to the vertex at which the two strings are concatenated. It is therefore convenient to assign waiting time $1:I_\Vv\to H$ to the outer vertices $\eps_0$ and $\eps_{k+1}$.

The structure of $W_k(H,\Iso)$ is determined inductively, by saying that for $k=0$, the two objects $W_0(H,\Iso)(\eps_0,\eps_1)$ are the unit $I_\Vv=H^{\otimes 0}$, while for $k>0$, the two $k$-cubes are attached to $W_{k-1}(H,\Iso)$ according to the following subdivided pushouts\begin{diagram}[small]H_-^{\otimes k}&\rTo&W_{k-1}(H,\Iso)(\eps_0,\eps_{k+1})\\\dTo&&\dTo\\H_\pm^{\otimes k}&\rTo&W^+_{k-1}(H,\Iso)(\eps_0,\eps_{k+1})\\\dTo&&\dTo\\H^{\otimes k}&\rTo&W_k(H,\Iso)(\eps_0,\eps_{k+1})\end{diagram}in which $H_-^{\otimes k}$ denotes the union of the $k$ faces of the $k$-cube $H^{\otimes k}$ obtained by inserting $0:I_\Vv\to H$ into each of the $k$ tensor factors of $H^{\otimes k}$, while $H_\pm^{\otimes k}$ denotes the whole boundary of the $k$-cube, i.e. the union of the $2k$ faces obtained by inserting $(0,1):I_\Vv\sqcup I_\Vv\ito H$ into each of the $k$ tensor factors of $H^{\otimes k}$. The attaching map $H_-^{\otimes k}\sqcup H_-^{\otimes k}\lrto W_{k-1}(H,\Iso)$ is defined by eliminating for each face the corresponding inner vertex from the string (this lowers the length of the string by $2$), and applying the monoid structure $H\otimes H\to H$ to the ``waiting times'' of the vertices surrounding the eliminated vertex. This definition is consistent precisely because $H$ is an interval in the sense of \cite[Definition 4.1]{BM1}.

Let us now consider a homotopy equivalence $\alpha:\JJ\to\AA_f$. The natural cofibration $\JJ\ito W(H,\Iso)$ identifies $\JJ(0,1)$ with $W_0(H,\Iso)(0,1)$, so that $\alpha$ amounts to the arrow $\alpha_0:I_\Vv=W_0(H,\Iso)(0,1)\to\AA_f(x,y)$. Any arrow $\beta_0:I_\Vv\to\AA_f(y,x)$ amounts to an extension of $\alpha$ to $W_0(H,\Iso)$. A further extension to $W_1(H,\Iso)$ amounts to a pair $(\alpha_1:H\to\AA_f(x,x),\beta_1:H\to\AA_f(y,y))$ of homotopies relating the identity of $x$ (resp. $y$) to the composite map $\beta_0\alpha_0$ (resp. $\alpha_0\beta_0$). Extending inductively $\alpha$ over $W_k(H,\Iso)$ amounts to the construction of ``higher homotopies'' $$\begin{cases}(\alpha_k:H^{\otimes k}\to\AA_f(x,y),\,\beta_k:H^{\otimes k}\to\AA_f(y,x))&\text{if }k\text{ is even,}\\(\alpha_k:H^{\otimes k}\to\AA_f(x,x),\,\beta_k:H^{\otimes k}\to\AA_f(y,y))&\text{if }k\text{ is odd,}\end{cases}$$satisfying certain coherence relations.

We now describe the precise inductive procedure to extend $\alpha$ to the whole $\Vv$-interval $W(H,\Iso)$. As in the proof of \cite[Lemma 5.4]{BM1} we inductively construct maps of $\Vv$-graphs $W_k(H,\Iso)\to\AA_f$ which are compatible with the \emph{partial $\Vv$-category} structure of $W_k(H,\Iso)$, cf. \cite[Definition 5.2]{BM1}. More precisely, the compatibility of $W_{k-1}(H,\Iso)\to\AA_f$ with the partial $\Vv$-category structure of $W_{k-1}(H,\Iso)$ allows a canonical extension along $W_{k-1}(H,\Iso)\ito W^+_{k-1}(H,\Iso)$. For the inductive step it then remains to be shown that the induced maps$$\begin{cases}H_\pm^{\otimes k}\to\AA_f(x,y),\text{ resp. }H_\pm^{\otimes k}\to\AA_f(y,x)&\text{if }k\text{ is even,}\\H_\pm^{\otimes k}\to\AA_f(x,x),\text{ resp. }H_\pm^{\otimes k}\to\AA_f(y,y)&\text{if }k\text{ is odd,}\end{cases}$$may be extended to the whole $k$-cube $H^{\otimes k}$, thus defining the higher homotopies $\alpha_k$, resp. $\beta_k$, together with the required extension along $W_{k-1}^+(H,\Iso)\ito W_k(H,\Iso)$. It turns out that in order to keep track of the necessary coherence relations it is best to construct the pair $(\alpha_k,\beta_{k-1})$ in parallel, assuming inductively that $(\alpha_j,\beta_{j-1})$ have already been defined for $j<k$. We will treat the case of odd $k$ explicitly, and leave the similar case of even $k$ to the reader.

We shall use Lemma \ref{parallellifting} as well as a suitable decomposition of the boundary $H_\pm^{\otimes k}$ of the $k$-cube $H^{\otimes k}$. For odd $k$, the string corresponding to this $k$-cube has the following form\begin{diagram}[small]0&\rTo^{\phi}&1&\rTo^{\psi}&0&\rTo&\cdots&\rTo&1&\rTo^{\psi}&0\end{diagram}where $\phi$ is taken to $\alpha_0:I_\Vv\to\AA_f(x,y)$ and $\psi$ is taken to $\beta_0:I_\Vv\to\AA_f(y,x)$. We shall denote by $F$ the face obtained by assigning waiting time $1$ to the \emph{first} inner vertex of the string. We shall denote by $L$ the union of the remaining $2k-1$ faces. $F$ is a $(k-1)$-cube $H^{\otimes (k-1)}$. Its boundary $\partial F=H^{\otimes(k-1)}_\pm$ embeds canonically into $F$ and into $L$. We thus get the following commutative diagram in $\Vv$:\begin{diagram}[small,UO]&&\AA_f(y,x)&&\\j\!\!\!\!&\ruTo&&\rdTo&\!\!\!\!{(\alpha_0)^*}\\\partial F&\rTo&L&\rTo^l&\AA_f(x,x)\\\dTo^\gamma&&\dTo_{\gamma'}&&\\F&\rTo_\delta&H^{\otimes k}&&\end{diagram}
Since $\alpha:\JJ\to\AA_f$ is a homotopy equivalence, precomposition with $\alpha_0$ acts as a weak equivalence. The maps $j$ and $l$ are defined by induction hypothesis. Since $H$ is an interval, $\gamma$ is a cofibration between cofibrant objects, $\gamma'$ is a trivial cofibration between cofibrant objects. Moreover, $(F\cup_{\partial F}L)=H^{\otimes k}_\pm\to H^{\otimes k}$ is a cofibration. We thus get by Lemma \ref{parallellifting} a pair of liftings$$(\beta_{k-1}:F=H^{\otimes(k-1)}\to\AA_f(y,x),\,\alpha_k:H^{\otimes k}\to\AA_f(x,x))$$ as required for the inductive step.\end{proof}

\section{Enriched categories with two objects}\label{s3}

The goal of this section is to give complete proofs of the Interval Cofibrancy Theorem \ref{cofibrant} and the Interval Amalgamation Lemma \ref{IAL}, which are essential ingredients for the canonical model structure on $\VCat$ as we have seen.

\begin{stit}\emph{Notation.} For any $\Vv$-category $\HH$ with object set $\{0,1\}$, we write $\HH(i,j)$ for the hom-object (in $\Vv$) of maps from $i$ to $j$ in $\HH$, and abbreviate $\HH(i,i)$ to $\HH_i$. Moreover, we write$$\partial\HH_1=\HH(0,1)\otimes_{\HH_0}\HH(1,0)$$and$$\partial\HH_0=\HH(1,0)\otimes_{\HH_1}\HH(0,1).$$Here, we use that $\HH(0,1)$ has compatible right $\HH_0$- and left $\HH_1$-actions, i.e. $\HH(0,1)$ is a \emph{$\HH_1$-$\HH_0$-bimodule}. Symmetrically, $\HH(1,0)$ is a \emph{$\HH_0$-$\HH_1$-bimodule}. The tensors $\otimes_{\HH_0}$ and $\otimes_{\HH_1}$ are defined as usual by certain coequalizers involving the monoidal structure of $\Vv$. Composition in $\HH$ induces maps $\partial\HH_1\to\HH_1$ and $\partial\HH_0\to\HH_0$.

Conversely, the structure of a $\Vv$-category on $\{0,1\}$ can be recovered from the sixtuple $(\HH_0,\HH_1,\HH(0,1),\HH(1,0),\partial\HH_1\overset{c_1}{\to}\HH_1,\partial\HH_0\overset{c_0}{\to}\HH_0)$ consisting of two monoids $\HH_0,\HH_1$, an $\HH_1$-$\HH_0$-bimodule $\HH(0,1)$, an $\HH_0$-$\HH_1$-bimodule $\HH(1,0)$, and two maps of bimodules $c_1$ and $c_0$ which satisfy the following compatibility relations:\begin{diagram}[small]\HH(0,1)\otimes_{\HH_0}\partial\HH_0&\rTo^\cong&\partial\HH_1\otimes_{\HH_1}\HH(0,1)&\quad\quad&\HH(1,0)\otimes_{\HH_1}\partial\HH_1&\rTo^\cong&\partial\HH_0\otimes_{\HH_0}\HH(1,0)\\\dTo^{id\otimes c_0}&&\dTo_{c_1\otimes id}&\quad\quad&\dTo^{id\otimes c_1}&&\dTo_{c_0\otimes id}\\\HH(0,1)\otimes_{\HH_0}\HH_0&\rTo_\cong&H_1\otimes_{\HH_1}\HH(0,1)&\quad\quad&\HH(1,0)\otimes_{\HH_1}\HH_1&\rTo_\cong&H_0\otimes_{\HH_0}\HH(1,0)\end{diagram}\end{stit}

The following slightly more elaborate version of Theorem \ref{cofibrant} will be established by a transfinite induction in which part (iii) plays an essential role.

\begin{thm}\label{mainthm}Assume that $\Vv$ is an adequate monoidal model category with cofibrant unit. Let $\HH$ be a cofibrant $\Vv$-category in $\VCat_{\{0,1\}}$. Then\begin{itemize}\item[(i)]$\HH_0$ and $\HH_1$ are cofibrant monoids;\item[(ii)]$\HH(0,1)$ is cofibrant as a right $\HH_0$-module and as a left $\HH_1$-module;\\$\HH(1,0)$ is cofibrant as a right $\HH_1$-module and as a left $\HH_0$-module;\item[(iii)]The maps $\partial\HH_0\to\HH_0$ and $\partial\HH_1\to\HH_1$ are cofibrations between cofibrant objects in $\Vv$.\end{itemize}\end{thm}

Special cases of this theorem are known. For instance, if $\Vv$ is the category of simplicial sets, then (i) was proved by Dwyer-Kan \cite{DK2}, and used by Bergner \cite{Be1} in her proof of the canonical model structure on simplicially enriched categories. It is natural to ask whether our methods extend, to prove a more general theorem for $\Vv$-categories on an arbitrary fixed object set $S$. We have not investigated this.

\subsection{Some excellent Quillen pairs}\label{examples}--\vspace{1ex}

We assume throughout that $\Vv$ is an \emph{adequate monoidal model category with cofibrant unit}. As mentioned in the introduction, adequacy implies the existence of a transferred Quillen model structure on ``structured objects'' in $\Vv$. Some instances of this are important for the proof of Theorem \ref{mainthm} and we discuss them now. It turns out that the corresponding Quillen pairs (formed by the free and forgetful functors) have the property that the right adjoint not only preserves \emph{and reflects} weak equivalences and fibrations (as in any transfer) but also \emph{preserves cofibrant objects}. Any Quillen pair with such a right adjoint will be called \emph{excellent}. It follows immediately from this definition that excellent Quillen pairs compose. Note that in establishing that the Quillen pairs below are excellent, it is essential that the unit $I_\Vv$ is cofibrant in $\Vv$.\vspace{1ex}

(a) Let $R$ be a monoid in $\Vv$, and assume that $R$ is well-pointed (i.e. the unit $I_\Vv\to R$ is a cofibration in $\Vv$). Then the categories ${}_R\Mod$ and $\Mod_R$ of left and right $R$-modules both admit a transferred model structure and the forgetful functors are part of excellent Quillen pairs.\vspace{1ex}

(b) For well-pointed monoids $R$ and $S$, the category ${}_R\Mod_S$ of $R$-$S$-bimodules admits a transferred model structure. The forgetful functor is again part of an excellent Quillen pair. This follows from the previous example by considering the monoid $R\otimes S^\op$.\vspace{1ex}

(c) The category $\Mon_\Vv$ of monoids in $\Vv$ admits a transferred model structure and the free-forgetful adjunction  $T:\Vv\lra\Mon_\Vv:U$ is an excellent Quillen pair. The preservation of cofibrant objects and cofibrations between cofibrant objects under the forgetful functor follows either from \cite[Corollary 5.5]{BM0} or, more directly, from the explicit construction of free monoid extensions as described by Rezk-Schwede-Shipley \cite{SS}. We review their construction in some detail here since similar constructions will be used repeatedly in the proof of Theorem \ref{mainthm}.

Let $R$ be a monoid in $\Vv$, and let $u:Y_0\to Y_1$ be a map in $\Vv$ equipped with a map $Y_0\to U(R)$. The \emph{free monoid extension $R[u]$ generated by $u$} is defined by the following pushout in monoids
\begin{gather*}\begin{diagram}[small]T(Y_0)&\rTo&R\\\dTo&&\dTo\\T(Y_1)&\rTo&\NWpbk R[u]\end{diagram}\end{gather*}in which the upper horizontal arrow is adjoint to the given $Y_0\to U(R)$. The crucial observation of Rezk-Schwede-Shipley \cite{SS} is that this pushout in monoids can be realised as a sequential colimit of pushouts in the category ${}_R\Mod_R$ of $R$-bimodules.

If $F_R:\Vv\to{}_R\Mod_R$ denotes the free $R$-bimodule functor, then the construction goes as follows. Let $R[u]^{(0)}=R$ and define inductively $R[u]^{(n)}$ by the following pushout in $R$-bimodules
\begin{gather*}\begin{diagram}[small]Y^{(n)}_-&\rTo&R[u]^{(n-1)}\\\dTo&&\dTo\\Y^{(n)}&\rTo&\NWpbk R[u]^{(n)}\end{diagram}\end{gather*}
where $$Y^{(n)}=\overbrace{F_R(Y_1)\otimes_R\cdots\otimes_RF_R(Y_1)}^n=R\otimes \overbrace{Y_1\otimes R\otimes\cdots\otimes Y_1\otimes R}^n$$
and $Y_-^{(n)}$ is the colimit of a diagram over a punctured $n$-cube $\{0,1\}^n-\{(1,\dots,1)\}$ in which the vertex $(i_1,\dots,i_n)$ takes the value $F_R(Y_{i_1})\otimes_R\cdots\otimes_RF_R(Y_{i_n})$ and the edge-maps are induced by $F_R(u)$. The map $Y^{(n)}_-\to Y^{(n)}$ is the comparison map from the colimit of this diagram to the value at $(1,\dots,1)$ of the extended diagram on the whole $n$-cube. The map $Y^{(n)}_-\to R[u]^{(n-1)}$ is defined inductively, using the fact that the construction of $R[u]^{(n-1)}$ involves $n-1$ tensor factors $F_R(Y_1)$ only.

Since the tensor $-\otimes_R-$ commutes with pushouts in both variables, there are canonical maps of $R$-bimodules $R[u]^{(p)}\otimes_R R[u]^{(q)}\to R[u]^{(p+q)}$. Since the tensor $-\otimes_R-$ commutes with countable sequential colimits in both variables, these maps induce the structure of a monoid on the colimit $R[u]=\lim_nR[u]^{(n)}$. It is straightforward to check that this monoid has indeed the required universal property.

Now, any cofibrant monoid is constructed out of the inital monoid $I_\Vv$ by transfinite composition of free monoid extensions along $T(u)$ where $u$ is a generating cofibration in $\Vv$, and taking retracts thereof. Assuming inductively that $R$ has an underlying cofibrant object (which we can, since $I_\Vv$ is cofibrant in $\Vv$), the pushout-product axiom implies that $Y_-^{(n)}\to Y^{(n)}$ and hence $R[u]^{(n-1)}\to R[u]^{(n)}$ are cofibrations in $\Vv$. It follows that $R\to R[u]$ is a cofibration in $\Vv$ so that (by induction) any cofibrant monoid has an underlying cofibrant object. Note that a similar argument shows that the forgetful functor takes any cofibration between cofibrant monoids to a cofibration between cofibrant objects in $\Vv$.\vspace{1ex}

(d) For a monoid $R$ in $\Vv$ let $\Alg_R$ be the category of monoids in $\Vv$ under $R$. This category inherits a model structure as an undercategory of the preceding example. There is an obvious forgetful functor $U_R:\Alg_R\to {}_R\Mod_R$ whose left adjoint $T_R:{}_R\Mod_R\to \Alg_R$ is given by$$T_R(N)=\coprod_{n\geq 0}\overbrace{N\otimes_R\cdots\otimes_RN}^n$$and this adjoint pair is a Quillen pair. The model structure on $\Alg_R$ coincides with the one obtained by transfer along this adjoint pair from the transferred model structure on $R$-bimodules. As in the preceding example, $T_R$-free extensions in $\Alg_R$ can be computed as sequential colimits of pushouts in ${}_R\Mod_R$. Note that $(T_R,U_R)$ is an excellent Quillen pair if and only if $R$ is cofibrant as an $R$-bimodule.\vspace{1ex}

(e) Let $S$ be a fixed set of objects. Then the category $\VCat_S$ of $\Vv$-categories with fixed object set $S$ admits a transferred model structure with generating (trivial) cofibrations of the form$$\JJ_{s,t}[X]\to\JJ_{s,t}[Y]$$where $X\to Y$ is a generating (trivial) cofibration in $\Vv$, and $s,t\in S$. The $\Vv$-category $\JJ_{s,t}[X]$ is characterized by the universal property that maps $\JJ_{s,t}[X]\to\AA$ into any $\Vv$-category $\AA$ on $S$ are in bijective correspondence with maps $X\to\AA(s,t)$ in $\Vv$. If $s\not=t$, then $\JJ_{s,t}[X]$ has $X$ as hom-object from $s$ to $t$, and only identity arrows elsewhere; if $s=t$, $\JJ_{s,t}[X]$ has the free monoid on $X$ as endo-hom-object of $s$ and identity arrows elsewhere. For each $s\in S$ there is a forgetful functor $$(-)_s:\VCat_S\to\Mon_\Vv$$mapping $\AA$ to the endomorphism-monoid $\AA_s$. For each pair $s,t$ of elements of $S$, there is a forgetful functor$$(-)(s,t):\VCat_S\to{}_{\AA_t}\Mod_{\AA_s}$$ mapping $\AA$ to the $\AA_t$-$\AA_s$-bimodule $\AA(s,t)$. Both functors are right Quillen functors.

These right adjoints, when composed with the appropriate forgetful functors to $\Vv$, preserve cofibrant objects. More precisely, the adjoint pair $$\cat_S:\VGrph_S\lra\VCat_S:U_S$$ is an excellent Quillen pair, where a \emph{$\Vv$-graph} $\AA$ on $S$ is  by definition a doubly indexed family $(\AA(s,t))_{(s,t)\in S^2}$ of objects of $\Vv$. The model structure on $\VGrph_S$ is the one induced from $\Vv$ through the isomorphism $\VGrph_S\cong\Vv^{S^2}$. The forgetful functor $U_S$ takes a $\Vv$-category to the obvious underlying $\Vv$-graph. It is well-known that this forgetful functor preserves filtered colimits, and that $\Vv$-categories on $S$ can be identified with \emph{monoids} in $\VGrph_S$ with respect to the following \emph{circle-product}:$$(\BB\circ\AA)(r,t)=\coprod_{s\in S}\BB(s,t)\otimes\AA(r,s)$$This circle-product is not symmetric, but satisfies all formal properties needed to construct free $\circ$-monoid extensions as sequential colimits in $\VGrph_S$, exactly like in the one object case treated in Section \ref{examples}c, cf. Schwede-Shipley \cite[Section 6.2]{SS2}. In particular, an induction on the construction of cofibrant objects in $\VCat_S$ shows that the forgetful functor $U_S$ preserves cofibrant objects. Theorem \ref{mainthm} considerably refines this preservation property of $U_S$ in the case $S=\{0,1\}$.\vspace{1ex}


The following two lemmas are preparatory for the proof of Theorem \ref{mainthm}.

\begin{lma}\label{cofibrant2}Let $R$ be a well-pointed monoid in a monoidal model category $\,\Vv$ with cofibrant unit, and let $M$ (resp. $N$) be a cofibrant right (resp. left) $R$-module. Then, the functors $M\otimes_R-:{}_R\Mod\to\Vv$ and $-\otimes_RN:\Mod_R\to\Vv$ are left Quillen functors. In particular, the object $M\otimes_RN$ is cofibrant in $\Vv$.
\begin{proof}The second assertion follows from the first. For the first, note that the underlying objects of $M$ and $N$ are cofibrant, cf. Section \ref{examples}a, so that the right adjoints of $M\otimes_R-$ and of $-\otimes_RN$ preserve fibrations and trivial fibrations by the adjoint form of the pushout-product axiom.\end{proof}\end{lma}

\begin{lma}\label{bimoduleadjoint}For any $\Vv$-category $\HH$ with object-set $\{0,1\}$ and any morphism of monoids $\HH_0\to\KK_0$, the following pushout in $\HH_1$-bimodules\begin{diagram}[small]\HH(0,1)\otimes_{\HH_0}\HH_0\otimes_{\HH_0}\HH(1,0)&\rTo&\HH_1\\\dTo&&\dTo\\\HH(0,1)\otimes_{\HH_0}\KK_0\otimes_{\HH_0}\HH(1,0)&\rTo&\KK_1\end{diagram}
endows $\KK_1$ with a canonical structure of monoid under $\HH_1$.

\begin{proof}The unit of $\KK_1$ is the composite $I_\Vv\to\HH_1\to\KK_1$, the multiplication $\KK_1\otimes_{\HH_1}\KK_1\to\KK_1$ is induced by pasting together the left and right $\HH_1$-module structures of $\HH(0,1)\otimes_{\HH_0}\KK_0\otimes_{\HH_0}\HH(1,0)$, the monoid structure of $\HH_1$ and the following map:\begin{diagram}[small](\HH(0,1)\otimes_{\HH_0}\KK_0\otimes_{\HH_0}\HH(1,0))\otimes_{\HH_1}(\HH(1,0)\otimes_{\HH_0}\KK_0\otimes_{\HH_0}\HH(0,1))\\\dTo^\cong\\
\HH(0,1)\otimes_{\HH_0}\KK_0\otimes_{\HH_0}\partial\HH_0\otimes_{\HH_0}\KK_0\otimes_{\HH_0}\HH(0,1)\\\dTo\\
\HH(0,1)\otimes_{\HH_0}\KK_0\otimes_{\HH_0}\HH_0\otimes_{\HH_0}\KK_0\otimes_{\HH_0}\HH(0,1)\\\dTo\\
\HH(0,1)\otimes_{\HH_0}\KK_0\otimes_{\HH_0}\HH(0,1)\end{diagram}\end{proof}\end{lma}

\subsection{Proof of Theorem \ref{mainthm}}\label{proofICT}--\vspace{1ex}

As usual, the cofibrant objects of $\VCat_{\{0,1\}}$ are built up from the initial object by taking (possibly transfinite) compositions of pushouts along generating cofibrations, and retracts thereof. We will prove the theorem by induction on the construction of $\HH$. We will be careful to establish the necessary properties for transfinite composition in the inductive step.

Note that the initial $\Vv$-category on $\{0,1\}$ has the properties stated in the theorem since the unit $I_\Vv$ is supposed to be cofibrant in $\Vv$. Let us begin by checking that the properties stated in the theorem are preserved under retract. If$$i:\HH\lra\KK:r$$ makes $\HH$ a retract of $\KK$, then the monoids $\HH_0$ and $\HH_1$ are retracts of the monoids $\KK_0$ and $\KK_1$. Moreover, $i$ and $r$ give maps of monoids $i_0:\HH_0\lra\KK_0:r_0$ and maps$$\alpha:\HH(0,1)\to i_0^*\KK(0,1)\quad\textrm{and}\quad\beta:\KK(0,1)\to r_0^*\HH(0,1)$$of right $\HH_0$-modules (respectively, $\KK_0$-modules), the transposed maps of which,$$\bar{\alpha}:i_{0!}\HH(0,1)\to\KK(0,1)\quad\textrm{and}\quad\bar{\beta}:r_{0!}\KK(0,1)\to\HH(0,1)$$make $\HH(0,1)$ into a retract of $r_{0!}\KK(0,1)$ as right $\HH_0$-modules, as in
\begin{diagram}[small,UO,silent]\HH(0,1)&\rTo&r_{0!}i_{0!}\HH(0,1)\\{\bar{\beta}}\!\!\!\!&\luTo&\dTo_{r_{0!}\bar{\alpha}}\\&&r_{0!}\KK(0,1).\end{diagram}
The other cases in the statement of the theorem are treated similarly.

It thus remains to be shown that in a pushout square\begin{gather}\label{poutij}\begin{diagram}[small]\JJ_{i,j}[X]&\rTo&\HH\\\dTo^{\JJ_{i,j}[u]}&&\dTo\\\JJ_{i,j}[Y]&\rTo&\KK\end{diagram}\end{gather}in $\VCat_{\{0,1\}}$, if the properties stated in the theorem hold for $\HH$, then they hold for $\KK$. Moreover, to be able to analyse the transfinite composition, we need to show that for each generating cofibration $u:X\to Y$ the map $\HH\to\KK$ in such a pushout square induces cofibrations between the different underlying structures mentioned in the theorem. For instance, $\HH_1\to\KK_1$ has to be a cofibration of monoids and $\KK_1\otimes_{\HH_1}\HH(0,1)\to\KK(0,1)$ a cofibration of left $\KK_1$-modules. We will treat separately the two cases $i=0,j=1$ and $i=0=j$. The other two cases $i=1,j=0$ and $i=1=j$ can be dealt with symmetrically (or follow by replacing $\HH$ with its opposite category). We warn the reader that the proof is quite involved in each case, since we are going to construct the relevant pushouts explicitly.

\begin{stit}\emph{Explicit construction of the pushout (\ref{poutij}) in case $i=0,j=1$}.\vspace{1ex}

  We first give a formal definition of $(\KK_0,\KK_1,\KK(0,1),\KK(1,0))$, then add a more informal ``set-theoretical'' explanation of this definition, and a finally verify that the resulting $\Vv$-category $\KK$ has the properties stated in the theorem.\vspace{1ex}

\underline{Construction of $\KK_0$ and $\KK_1$}. Consider the cofibration of $\HH_0$-bimodules$$\HH(1,0)\otimes u\otimes\HH_0:\HH(1,0)\otimes X\otimes\HH_0\to\HH(1,0)\otimes Y\otimes\HH_0,$$and apply the free monoid functor $T_{\HH_0}$ of Section \ref{examples}(d) to it, to obtain a cofibration of monoids under $\HH_0$,\begin{gather}T_{\HH_0}(\HH(1,0)\otimes X\otimes\HH_0)\to T_{\HH_0}(\HH(1,0)\otimes Y\otimes\HH_0).\end{gather}
The given map $X\to\HH(0,1)$ together with composition in $\HH$ induces a canonical map $T_{\HH_0}(\HH(1,0)\otimes X\otimes\HH_0)\to\HH_0$, and $\KK_0$ is the pushout of monoids,\begin{gather}\label{pout01a}\begin{diagram}[small]T_{\HH_0}(\HH(1,0)\otimes X\otimes\HH_0)&\rTo&\HH_0\\\dTo&&\dTo\\T_{\HH_0}(\HH(1,0)\otimes Y\otimes\HH_0)&\rTo&\KK_0.\end{diagram}\end{gather}
The construction of $\KK_1$ is symmetric, as a pushout of monoids,\begin{gather}\label{pout01b}\begin{diagram}[small]T_{\HH_1}(\HH_1\otimes X\otimes\HH(1,0))&\rTo&\HH_1\\\dTo&&\dTo\\T_{\HH_1}(\HH_1\otimes Y\otimes\HH(1,0))&\rTo&\KK_1.\end{diagram}\end{gather}Both maps of monoids $\HH_0\to\KK_0$ and $\HH_1\to\KK_1$ are thus cofibrations of monoids.\vspace{1ex}

\underline{Construction of $\KK(0,1)$ and $\KK(1,0)$}. For a generating cofibration $X\to Y$ the pushout-product axiom yields cofibrations $(\HH_1\otimes X)\cup(\partial\HH_1\otimes Y)\to \HH_1\otimes Y$ and $(X\otimes\HH_0)\cup(Y\otimes\partial\HH_0)\to Y\otimes\HH_0$. Tensoring the first (resp. second) with $\KK_0$ (resp. $\KK_1$) from the right (resp. left) gives rise to the following two pushouts, of right $\KK_0$-modules, resp. left $\KK_1$-modules (these are calculated in $\Vv$)

\begin{gather}\label{pout01c}\begin{diagram}[small](\HH_1\otimes X\otimes\KK_0)\cup(\partial\HH_1\otimes Y\otimes\KK_0)&\rTo&\HH(0,1)\otimes_{\HH_0}\KK_0\\\dTo&&\dTo\\\HH_1\otimes Y\otimes\KK_0&\rTo& P\end{diagram}\end{gather}

\begin{gather}\label{pout01d}\begin{diagram}[small](\KK_1\otimes X\otimes\HH_0)\cup(\KK_1\otimes Y\otimes\partial\HH_0)&\rTo&\KK_1\otimes_{\HH_1}\HH(0,1)\\\dTo&&\dTo\\\KK_1\otimes Y\otimes\HH_0&\rTo& Q\end{diagram}\end{gather}
in which the upper vertical arrows are induced from the given map $X\to\HH(0,1)$ and the definitions of $\KK_0$ and $\KK_1$ respectively. We claim that $P$ and $Q$ are in fact isomorphic, and in particular carry a $\KK_1$-$\KK_0$-bimodule structure; moreover, this object $P\cong Q$ defines $\KK(0,1)$. To see this isomorphism, note that there are canonical maps $P\to Q$ and $Q\to P$ definable on the upper right and lower left corners of the pushouts. It can be checked that the two maps are mutually inverse.\vspace{1ex}

Finally, the two tensor products\begin{gather}\label{tensor01}\KK_0\otimes_{\HH_0}\HH(1,0)\quad\textrm{and}\quad\HH(1,0)\otimes_{\HH_1}\KK_1\end{gather}are isomorphic, and define $\KK(1,0)$.\vspace{1ex}

\underline{Category structure}. Clearly $\KK_0$ and $\KK_1$ are monoids and $\KK(0,1),\KK(1,0)$ are bimodules. This takes care of most of the category structure of $\KK$, except the compositions$$\KK(1,0)\otimes_{\KK_1}\KK(0,1)\to\KK_0\quad\textrm{and}\quad\KK(0,1)\otimes_{\KK_0}\KK(1,0)\to\KK_1.$$The first of these is most easily described as the ``obvious'' map$$(\HH(1,0)\otimes_{\HH_1}\KK_1)\otimes_{\KK_1}Q\to\KK_0$$and the second as $$P\otimes_{\KK_0}(\KK_0\otimes_{\HH_0}\HH(1,0))\to\KK_1.$$One can now check that $\KK$ is a well-defined $\Vv$-category, having the universal property of the pushout (\ref{poutij}) for $i=0$ and $j=1$.\vspace{1ex}

\underline{Informal description}. Set-theoretically, an element of $\KK_0$ is represented by a string (with $n\geq 0$)\begin{diagram}[small]0&\lTo^{h_1}&1&\lTo^{y_1}&0&\lTo^{h_2}&1&\lTo^{y_2}&0&\cdots\cdots&0&\lTo^{h_n}&1&\lTo^{y_n}&0&\lTo^f&0\end{diagram}where $f\in\HH_0,\,h_i\in\HH(1,0)$ and $y_i\in Y$. If one of the $y_i$ lies in the `smaller'' object $X$, this string is identified with the shorter one obtained by composing $h_i$ and $h_{i+1}$ with the image of $y_i$ in $\HH(0,1)$. Observe that this kind of identification corresponds precisely to analysing a pushout of \emph{monoids} like (\ref{pout01a}). The set-theoretical description of (\ref{pout01b}) is similar and uses strings of the form\begin{diagram}[small]1&\lTo^g&1&\lTo^{y_1}&0&\lTo^{h_1}&1&\lTo^{y_2}&0&\lTo^{h_2}&1&\cdots\cdots&1&\lTo^{y_n}&0&\lTo^{h_n}&1\end{diagram}
An arrow in $\KK(0,1)$ is either of the form\begin{diagram}[small]1&\lTo^{\xi}&1&\lTo^y&0&\lTo^h&0&\quad(\xi\in\KK_1,\,y\in Y,\,h\in\HH_0)\end{diagram}or of the form\begin{diagram}[small]1&\lTo^\xi&1&\lTo^h&0&\quad(\xi\in\KK_1,\,h\in\HH(0,1))\end{diagram}If in the first presentation $h$ is decomposable or $y\in X$ then it can be written as in the second form. This is the meaning of pushout (\ref{pout01c}). Note that if in the second form $\xi$ is decomposable then  $1\overset{\xi}{\lto}1\overset{h}{\lto}0$ can be written either as\begin{diagram}[small]1&\lTo^{\xi'}&1&\lTo^y&0&\lTo^{h'}&1&\lTo^{h}&0\end{diagram}or as \begin{diagram}[small]1&\lTo^{\xi'}&1&\lTo^{h''}&0&\lTo^{h'}&1&\lTo^{h}&0\end{diagram}The first one belongs to the image of $\KK_1\otimes Y\otimes\partial\HH_0$, the second one is equated with $1\overset{\xi'}{\lto}1\overset{h''h'h}{\lto}0$ because the tensor is over $\HH_1$.

Let us try to see set-theoretically why the pushouts $P$ and $Q$ (in the definition of $\KK(0,1)$) are isomorphic. We will describe the map $Q\to P$. The description of its inverse is symmetric. The map $Q\to P$ is defined on both corners of the pushout (\ref{pout01d}) as follows. An element of $\KK_1\otimes Y\otimes\HH_0$ looks like\begin{diagram}[small](1&\lTo^g&1&\lTo^{y_1}&0&\lTo^{h_1}&1&\cdots&1&\lTo^{y_n}&0&\lTo^{h_n}&1)\otimes(1&\lTo^y&0)\otimes(0&\lTo^h&0)\end{diagram}For $n>0$, it could equally well be ``parsed'' as
\begin{diagram}[small](1&\lTo^g&1)\otimes(1&\lTo^{y_1}&0)\otimes(0&\lTo^{h_1}&1&\lTo^{y_2}&0&\cdots&0&\lTo^{h_n}&1&\lTo^{y}&0&\lTo^h&0)\end{diagram}which is a typical element of $\HH_1\otimes Y\otimes\KK_0$ (of course one has to check that this is well-defined and corresponds to the diagrammatic definition).

If $n=0$, we have an element of $\HH_1\otimes Y\otimes\HH_0$,\begin{diagram}[small](1&\lTo^g&1)\otimes(1&\lTo^z&0)\otimes(0&\lTo^h&0)\end{diagram}which can be viewed as an element of $\HH_1\otimes Y\otimes\KK_0$. This describes the map $\KK_1\otimes Y\otimes\HH_0\to P$. The map $\KK_1\otimes_{\HH_1}\HH(0,1)\to P$ can be described as follows: a typical element of $\KK_1\otimes_{\HH_1}\HH(0,1)$ looks like\begin{diagram}[small](1&\lTo^g&1&\lTo^{y_1}&0&\lTo^{h_1}&1&\cdots&1&\lTo^{y_n}&0&\lTo^{h_n}&1)\otimes(1&\lTo^f&0)\end{diagram}and could be rewritten as\begin{diagram}[small](1&\lTo^g&1)\otimes(1&\lTo^{y_1}&0)\otimes(0&\lTo^{h_1}&1&\lTo^{y_2}&0&\cdots&0&\lTo^{h_n}&1&\lTo^{f}&0)\end{diagram}in $\HH_1\otimes Y\otimes\KK_0$ if $n>0$, and as $1\overset{gf}{\lto}0$ in $\HH(0,1)$, hence in $\HH(0,1)\otimes_{\HH_0}\KK_0$, if $n=0$. Together, these give the map $Q\to P$.

Finally, an element of $\KK(1,0)=\KK_0\otimes_{\HH_0}\HH(1,0)$ looks like $0\overset{\xi}{\lto}0\overset{h}{\lto}1$ where $\xi\in\KK_0$ and $h\in\HH(1,0)$, or more explicitly\begin{diagram}[small](0&\lTo^{h_1}&1&\lTo^{y_1}&0&\cdots&0&\lTo^{h_n}&1&\lTo^{y_n}&0&\lTo^y&0)\otimes(0&\lTo^h&1)\end{diagram}where in fact we can always assume $y=1$ because the tensor is over $\HH_0$, so really\begin{diagram}[small](0&\lTo^{h_1}&1&\lTo^{y_1}&0&\cdots&0&\lTo^{h_n}&1&\lTo^{y_n}&0)\otimes(0&\lTo^h&1)\end{diagram}which for $n>0$ can be rewritten as\begin{diagram}[small](0&\lTo^{h_1}&1)\otimes(1&\lTo^{y_1}&0&\cdots&0&\lTo^{h_n}&1&\lTo^{y_n}&0&\lTo^h&1)\end{diagram}a typical element of $\HH(1,0)\otimes_{\HH_1}\KK_1$. For $n=0$, we get just an element of $\HH(0,1)$ on both sides of (\ref{tensor01}).

In terms of these set-theoretical string diagrams, the category structure of $\KK$ is given by concatenation of strings.\vspace{1ex}

\underline{Verification of the properties stated in the theorem}.\vspace{1ex}

The monoids $\KK_0$ and $\KK_1$ are cofibrant by construction, cf. the pushouts (\ref{pout01a}) and (\ref{pout01b}). Note that, in addition, the maps $\HH_0\to\KK_0$ and $\HH_1\to\KK_1$ are cofibrations of monoids, a property needed for analysing tranfinite compositions of such pushouts.

Next, since tensoring along a cofibration of monoids is a left Quillen functor, the two descriptions of $\KK(1,0)$ in (\ref{tensor01}) imply that $\KK(1,0)$ is cofibrant, both as a left $\KK_0$- and as a right $\KK_1$-module. Similarly, it follows from the two descriptions of $\KK(0,1)$ in (\ref{pout01c}) and (\ref{pout01d}) that $\KK(0,1)$ is cofibrant, both as a left $\KK_0$- and as a right $\KK_1$-module. Note again that the canonical map $\KK_1\otimes_{\HH_1}\HH(0,1)\to\KK(0,1)$ (resp. $\HH(0,1)\otimes_{\HH_0}\KK_0\to\KK(0,1)$) is a cofibration of left $\KK_1$- (resp. right $\KK_0$-) modules as required for analysing transfinite compositions of such pushouts.

Finally, we will check that $\partial\KK_0\to\KK_0$ is a cofibration between cofibrant objects in $\Vv$. A similar proof applies to $\partial\KK_1\to\KK_1$. The cofibrant monoid $\KK_0$ is cofibrant as an object of $\Vv$, since the forgetful functor $U:\Mon_\Vv\to\Vv$ preserves cofibrant objects, cf. Section \ref{examples}(a). Moroeover, $\partial\KK_0=\KK(1,0)\otimes_{\KK_1}\KK(0,1)$ is also cofibrant in $\Vv$, by Lemma \ref{cofibrant2}, since $\KK(1,0)$ and $\KK(0,1)$ are cofibrant as $\KK_1$-modules.

To see that $\partial\KK_0\to\KK_0$ is a cofibration, we filter $\KK_0$ as follows. Put $$T^p=(\HH(1,0)\otimes Y)^{\otimes p}\otimes \HH_0$$and let $T_-^p$ be the colimit of similar objects with at least one $Y$ replaced by an $X$. The generating cofibration $X\to Y$ induces a canonical map $T_-^p\to T^p$, which can be obtained by iterated application of the pushout-product axiom followed by a tensor with $\HH_0$, hence the map $T_-^p\to T^p$ is a cofibration of right $\HH_0$-modules. Now let $\KK_0^{(0)}=\HH_0$ and for $p>0$ define $\KK_0^{(p)}$ by the pushout\begin{diagram}[small]T^p_-&\rTo&\KK_0^{(p-1)}\\\dTo&&\dTo\\T^p&\rTo&\KK_0^{(p)}\end{diagram}The pushout in monoids (\ref{pout01a}) defining $\KK_0$ can also be constructed as a sequential colimit in $\Vv$ of these pushouts $\KK_0^{(p)}$, cf. Sections \ref{examples}(c)-(e). In a similar way, the object $\partial\KK_0$ can be constructed as a sequential colimit in $\Vv$, except that one starts with $(\partial\KK_0)^{(0)}=\partial\HH_0$ and continues with pushouts\begin{diagram}[small]T^p_-&\rTo&(\partial\KK_0)^{(p-1)}\\\dTo&&\dTo\\T^p&\rTo&(\partial\KK_0)^{(p)}\end{diagram}for $p>0$. Thus, we have a pushout square of cofibrations between cofibrant objects\begin{diagram}[small](\partial\KK_0)^{(p-1)}&\rTo&\KK_0^{(p-1)}\\\dTo&&\dTo\\(\partial\KK_0)^{(p)}&\rTo&\KK_0^{(p)}\end{diagram}for each $p>0$. Therefore, the colimit $\partial\KK_0\to\KK_0$ is a cofibration between cofibrant objects as well. Since this ladder of pushouts starts with $\partial\HH_0\to\HH_0$ and yields in the colimit $\partial\KK_0\to\KK_0$ we actually get the following more precise result. The square\begin{diagram}[small]\partial\HH_0&\rTo&\HH_0\\\dTo&&\dTo\\\partial\KK_0&\rTo&\KK_0\end{diagram}is a pushout square in $\Vv$ (indeed, in the category of $\HH_0$-bimodules). This property is needed to analyse transfinite compositions of pushouts of the form (1).\end{stit}

\begin{stit}\emph{Explicit construction of the pushout (\ref{poutij}) in case $i=0=j$}\label{section00}.\vspace{1ex}

Again we begin by giving explicit descriptions of the hom-objects $\KK_0$, $\KK_1$, $\KK(0,1)$, $\KK(1,0)$. These are easier than in the case $i=0,\,j=1$; however the verification of the properties as stated in the theorem will be more involved.\vspace{1ex}

\underline{Construction of $\KK_0$ and $\KK_1$}. The monoid $\KK_0$ is defined as a pushout in the category of monoids in $\Vv$ (where $T$ denotes the free monoid functor),\begin{gather}\label{pout00a}\begin{diagram}[small]T(X)&\rTo&\HH_0\\\dTo&&\dTo\\T(Y)&\rTo&\KK_0\end{diagram}\end{gather}\vspace{1ex}
while $\KK_1$ is defined as a pushout in the category of $\HH_1$-bimodules

\begin{gather}\label{pout00b}\begin{diagram}[small]\HH(0,1)\otimes_{\HH_0}\HH_0\otimes_{\HH_0}\HH(1,0)&\rTo&\HH_1\\\dTo&&\dTo\\\HH(0,1)\otimes_{\HH_0}\KK_0\otimes_{\HH_0}\HH(1,0)&\rTo&\KK_1\end{diagram}\end{gather}\vspace{1ex}

\underline{Construction of $\KK(0,1)$ and $\KK(1,0)$}. These are defined by

\begin{gather}\label{tensor00}\KK(0,1)=\HH(0,1)\otimes_{\HH_0}\KK_0\quad\textrm{and}\quad\KK(1,0)=\KK_0\otimes_{\HH_0}\HH(1,0)\end{gather}\vspace{1ex}

\underline{Category structure}. First notice that $\KK_1$ is indeed a monoid under $\HH_1$ by Lemma \ref{bimoduleadjoint}. Next, $\KK(0,1)$ is a right $\KK_0$-module by construction; it is a left $\KK_1$-module by ``amalgamation'' of the left $\HH_1$-action on $\HH(0,1)\otimes_{\HH_0}\KK_0$ and the left $\HH(0,1)\otimes_{\HH_0}\KK_0\otimes_{\HH_0}\HH(1,0)$-action on $\HH(0,1)\otimes_{\HH_0}\KK_0$ (given by composition in $\HH$ and multipliciation in $\KK_0$). Similarly, $\KK(1,0)$ has the structure of a left $\KK_0$- and right $\KK_1$-module. Finally, there are canonical maps$$\KK(1,0)\otimes\KK(0,1)\rto\KK_0\otimes\partial\HH_0\otimes\KK_0\rto\KK_0$$and$$\KK(0,1)\otimes\KK(1,0)\rto\HH(0,1)\otimes_{\HH_0}\KK_0\otimes_{\HH_0}\HH(1,0)\rto\KK_1$$defining the remaining compositions in $\KK$. One now checks that these maps all together define a category structure on $\KK$, and that $\KK$ thus constructed has the universal property of the pushout (\ref{poutij}) for $i=0=j$.\vspace{1ex}

\underline{Verification of the properties stated in the theorem}.\vspace{1ex}

Assuming that $\HH$ has these properties and that $X\to Y$ is a cofibration in $\Vv$, we now check that $\KK$ has these properties as well, again making sure that transfinite composition of such pushouts is possible. Some of these properties are obvious, viz.
\begin{itemize}\item[(a)]$\KK_0$ is a cofibrant monoid (indeed, $\HH_0\to\KK_0$ is a cofibration of monoids);\item[(b)]$\KK(0,1)$ (resp. $\KK(1,0)$) is a cofibrant right (resp. left) $\KK_0$-module;\item[(c)]$\partial\KK_1\to\KK_1$ is a cofibration between cofibrant objects.\end{itemize}
Indeed, $\partial\KK_1=\KK(0,1)\otimes_{\KK_0}\KK(1,0)=\HH(0,1)\otimes_{\HH_0}\KK_0\otimes_{\HH_0}\HH(1,0)$, so that pushout (\ref{pout00b}) can be rewritten as\begin{diagram}[small]\partial\HH_1&\rTo&\HH_1\\\dTo&&\dTo\\\partial\KK_1&\rTo&\KK_1\end{diagram} from which (c) immediately follows. It thus remains to be proved
\begin{itemize}\item[(d)]$\KK_1$ is a cofibrant monoid (indeed, $\HH_1\to\KK_1$ is a cofibration of monoids);\item[(e)]$\KK(0,1)$ (resp. $\KK(1,0)$) is a cofibrant left (resp. right) $\KK_1$-module;\item[(f)]$\partial\KK_0\to\KK_0$ is a cofibration between cofibrant objects.\end{itemize}\vspace{1ex}

\underline{Proof of (d)}. We use that $\partial\HH_0\to\HH_0$ is a cofibration and define the following filtration on $\KK_1$, cf. Sections \ref{examples}(c)-(e). Put for $p>0$: $$Y^{(p)}=\HH(0,1)\otimes Y\otimes\HH_0\otimes\cdots\otimes\HH_0\otimes Y\otimes\HH(1,0),$$with $p$ occurences of $Y$ and $p-1$ occurrences of $\HH_0$. Let $Y^{(p)}_-$ be the canonical colimit of objects like $Y^{(p)}$ where at least one of the $Y$'s is replaced by an $X$. By the pushout-product axiom and the fact that $\HH_0$, $\HH(0,1)$, $\HH(1,0)$ are cofibrant in $\Vv$, the map$$Y^{(p)}_-\to Y^{(p)}$$ is a cofibration in $\Vv$. This cofibration is of the form $$\HH(0,1)\otimes A\otimes\HH(1,0)\to\HH(0,1)\otimes B\otimes\HH(1,0)$$ for a cofibration $A\to B$ in $\Vv$, which implies that it is a cofibration of $\HH_1$-bimodules, because of the cofibrancy of $\HH(0,1)$ and $\HH(1,0)$ as $\HH_1$-modules. The filtration on $\KK_1$ is defined by$$\KK_1^{(0)}=\HH_1$$ and the pushouts\begin{diagram}[small]Y^{(p)}_-&\rTo&\KK_1^{(p-1)}\\\dTo&&\dTo\\Y^{(p)}&\rTo&\KK_1^{(p)}\end{diagram}along the maps $Y^{(p)}_-\rto Y^{(p-1)}\rto\KK_1^{(p-1)}$ induced by $X\to\HH_0$, for $p>0$. Note that this filtration$$\KK_1^{(0)}\rto\KK_1^{(1)}\rto\KK_1^{(2)}\rto\cdots$$is not a filtration by monoids, although it is a filtration by $\HH_1$-bimodules. The multiplication on $\KK_1$ restricts to\begin{gather}\label{pout00K1}\KK_1^{(p)}\otimes_{\HH_1}\KK_1^{(q)}\rto\KK_1^{(p+q)}.\end{gather}If $M$ is any monoid and $\HH_1\to M$ is a map of monoids (making $M$ into an $\HH_1$-bimodule) then we will call a map $\phi:\KK_1\to M$ \emph{multiplicative} if it is a map of $\HH_1$-bimodules which makes the diagram\begin{diagram}[small]\KK_1^{(p)}\otimes_{\HH_1}\KK_1^{(q)}&\rTo^{\phi\otimes\phi}&M\otimes_{\HH_1}M\\\dTo&&\dTo\\\KK_1^{(n)}&\rTo^\phi&M\end{diagram}commute, for any two $p,q>0$ with $p+q=n$. A compatible sequence $\phi^{(n)}:\KK_1^{(n)}\to M$ (compatible in the sense that $\phi^{(n)}$ precomposed with $\KK_1^{(n-1)}\to\KK_1^{(n)}$ is $\phi^{(n-1)}$) of multiplicative maps will define a map of monoids $\KK_1\to M$.

Consider again the object $Y^{(p)}$, for $p>1$, so that there is at least one occurence of $\HH_0$. Let $Y^{(p)}_\partial$ be the colimit of all similar objects where at least one occurrence of $\HH_0$ is replaced by $\partial\HH_0$. Then the cofibration $\partial\HH_0\to\HH_0$ induces a map $Y^{(p)}_\partial\to Y^{(p)}$, which is again a cofibration of $\HH_1$-bimodules by the pushout-product axiom, just like for $Y^{(p)}_-\to Y^{(p)}$ above. The inclusions of the different filtration stages of $\KK_1$ can be refined as in$$\cdots\rto\KK_1^{(p-1)}\rto\KK_{1,\partial}^{(p)}\rto\KK_1^{(p)}\rto\cdots$$where $\KK_{1,\partial}^{(p)}$ fits into pushouts\begin{diagram}[small]Y^{(p)}_-&\rTo&Y^{(p)}_-\cup Y^{(p)}_\partial&\rTo& Y^{(p)}\\\dTo&&\dTo&&\dTo\\K_1^{(p-1)}&\rTo&\KK_{1,\partial}^{(p)}&\rTo&\KK_1^{(p)}\end{diagram}in which all horizontal maps are cofibrations of $\HH_1$-bimodules. Moreover, the multiplication (\ref{pout00K1}) for $p,q>0$ with $p+q=n$ factors through $\KK_{1,\partial}^{(n)}\to\KK_1^{(n)}$. Thus, a map of $\HH_1$-bimodules $\KK_1^{(n)}\to M$ is multiplicative if and only if its restriction to $\KK_{1,\partial}^{(n)}$ is; in particular, it makes sense to say of a map $\KK_{1,\partial}^{(n)}\to M$ that it is multiplicative. In fact, any multiplicative map $\KK_1^{(n-1)}\rto M$ extends uniquely to a multiplicative map $\KK_{1,\partial}^{(n)}\rto M$ because, by definition, $\KK_{1,\partial}^{(n)}$ is the colimit of the diagram over $\KK_1^{(n)}$ given by all the maps (\ref{pout00K1}) with $p,q>0,\,p+q=n$.

With this notation, it is now easy to prove that $\HH_1\to\KK_1$ is a cofibration of monoids. Indeed, suppose we are given a commutative square\begin{diagram}[small]\HH_1&\rTo^\phi& M\\\dTo&&\dTo\\\KK_1&\rTo^\psi&N\end{diagram}of monoids where $M\rto N$ is a trivial fibration in the transferred model structure. The map $\HH_1\to M$ makes $M\to N$ into a map of $\HH_1$-bimodules, which is again a trivial fibration in the appropriate transferred model structure. We now construct a compatible sequence of multiplicative lifts$$\phi^{(n)}:\KK_1^{(n)}\to M$$making the appropriate diagrams commute, as follows. For $n=0$, we take $\phi^{(0)}$ to be $\phi$. Next we extend it to $\phi^{(1)}$ by lifting in the diagram of $\HH_1$-bimodules\begin{diagram}[small,UO,silent]\KK_1^{(0)}&\rTo&M\\\dTo&\ruDotsto&\dTo\\\KK_1^{(1)}&\rTo&N\end{diagram}which is possible since the left vertical map is a cofibration of $\HH_1$-bimodules. For $n>1$, we first extend $\phi^{(n-1)}:\KK_1^{(n-1)}\to M$, already found, uniquely to a multiplicative map $\phi_\partial^{(n)}:\KK_{1,\partial}^{(n)}\to M$. Next we use that $\KK_{1,\partial}^{(n)}\to\KK_1^{(n)}$ is a cofibration of $\HH_1$-bimodules, and extend $\phi_\partial^{(n)}$ to $\phi^{(n)}:\KK_1^{(n)}\to M$. This map of $\HH_1$-bimodules is automatically multiplicative, as noted above. The sequence $\phi^{(n)}$ thus found gives the required diagonal $\KK_1\to M$. This completes the proof that $\HH_1\to\KK_1$ is a cofibration of monoids, and hence proves (d).\vspace{1ex}

\underline{Proof of (e)}. We will show that $\KK(1,0)$ is a cofibrant right $\KK_1$-module. In fact, our proof will show that the canonical map $\HH(1,0)\otimes_{\HH_1}\KK_1\to\KK(1,0)$ is a cofibration of right $\KK_1$-modules. This stronger property is needed for analysing transfinite compositions of pushouts (\ref{poutij}). The proof that $\KK(0,1)$ is a cofibrant left $\KK_1$-module and, in fact, that $\KK_1\otimes_{\HH_1}\HH(0,1)\to\KK(0,1)$ is a cofibration of left $\KK_1$-modules, is similar.

Recall that $\KK_0$ is constructed from $\HH_0$ by the pushout of monoids (\ref{pout00a}) for a generating cofibration $X\to Y$. Thus $\KK_0$ is naturally filtered as follows. Let$$W^{(p)}=\HH_0\otimes Y\otimes\HH_0\otimes\cdots\otimes Y\otimes\HH_0\quad(p>0)$$ with $p$ occurences of $Y$, and let $W_-^{(p)}$ be the colimit of similar objects where at least one of the $Y$'s is replaced by an $X$, so that we have canonical cofibrations $W_-^{(p)}\to W^{(p)}$ as given by the pushout-product axiom. Let $\KK^{(0)}=\HH_0$, and for $p>0$, let $\KK_0^{(p)}$ be defined by the pushout\begin{gather}\label{pout00filt00}\begin{diagram}[small]W_-^{(p)}&\rTo&\KK_0^{(p-1)}\\\dTo&&\dTo\\W^{(p)}&\rTo&\KK_0^{(p)}\end{diagram}\end{gather}Then $\KK_0$ is the colimit of cofibrations$$\KK_0^{(0)}\rto\KK_0^{(1)}\rto\KK_0^{(2)}\rto\cdots$$Since by definition $\KK(1,0)=\KK_0\otimes_{\HH_0}\HH(1,0)$, the hom-object $\KK(1,0)$ has a similar filtration starting with $\KK(1,0)^{(0)}=\HH(1,0)$ and defined by successive pushouts\begin{gather}\label{pout00filt10}\begin{diagram}[small]W_-^{(p)}\otimes_{\HH_0}\HH(1,0)&\rTo&\KK(1,0)^{(p-1)}\\\dTo&&\dTo\\W^{(p)}\otimes_{\HH_0}\HH(1,0)&\rTo&\KK(1,0)^{(p)}\end{diagram}\end{gather}This is a filtration by right $\HH_1$-modules (not by $\KK_1$-modules). Note that the filtration of $\KK_1$ used in the proof of (d) has been constructed in an analogous way, starting with $\KK_1^{(0)}=\HH_1$, and using pushouts with left vertical maps of the form $\HH(0,1)\otimes_{\HH_0} W_-^{(p)}\otimes_{\HH_0}\HH(1,0)\rto\HH(0,1)\otimes_{\HH_0} W^{(p)}\otimes_{\HH_0}\HH(1,0)$, denoted $Y^{(p)}_-\rto Y^{(p)}$. In particular, this yields canonical filtered action maps$$\KK(1,0)^{(p)}\otimes_{\HH_1}\KK_1^{(q)}\to\KK(1,0)^{(p+q)}$$which in the colimit define the right $\KK_1$-action on $\KK(1,0)$.

To prove that $\KK(1,0)$ is a cofibrant right $\KK_1$-module, consider a trivial fibration of right $\KK_1$-modules $M\to N$, and a map $\KK(1,0)\to N$. We will construct a lift $\KK(1,0)\to M$ by successively lifting $\KK(1,0)^{(p)}\to\KK(1,0)\to N$ to $\KK(1,0)^{(p)}\to M$ as right $\HH_1$-module maps. To make sure that the resulting lift is a map of right $\KK_1$-modules, we need the $\phi^{(n)}$ to be multiplicative, in the sense that each square\begin{diagram}[small]\KK(1,0)^{(p)}\otimes_{\HH_1}\KK_1^{(q)}&\rTo^{\phi^{(p)}\otimes id}&M\otimes_{\HH_1}\KK_1^{(q)}\\\dTo&&\dTo\\\KK(1,0)^{(p+q)}&\rTo^{\phi^{(p+q)}}&M\end{diagram} commutes, for $p+q\leq n$. For $n=0$, we find $\phi^{(0)}:\KK(0,1)^{(0)}=\HH(1,0)\to M$ because $\HH(1,0)$ is cofibrant as a right $\HH_1$-module. In order to extend $\phi^{(n-1)}$ to $\phi^{(n)}$ it suffices to show that\begin{gather}\label{pout00cof}\bigcup_{p+q=n,p<n}\KK(1,0)^{(p)}\otimes_{\HH_1}\KK_1^{(q)}\rto\KK(1,0)^{(n)}\end{gather} is a cofibration of right $\HH_1$-modules. For $n>0$ fixed, we write $A$ for the domain of (\ref{pout00cof}) and represent $A$ as a union $\KK(1,0)^{(n-1)}\cup A'$ where$$A'=\bigcup_{0\leq k<n}(\HH_0\otimes Y)^{\otimes k}\otimes\partial\HH_0\otimes (Y\otimes\HH_0)^{\otimes n-k-1}\otimes Y\otimes\HH(1,0).$$By definition, $A'=U\otimes\HH(1,0)$ and $\KK(1,0)^{(n)}=\KK(1,0)^{(n-1)}\cup (V\otimes\HH(1,0))$ for\begin{align*}U=&\bigcup_{0\leq k<n}(\HH_0\otimes Y)^{\otimes k}\otimes\partial\HH_0\otimes (Y\otimes\HH_0)^{\otimes n-k-1}\otimes Y\quad\textrm{and}\\V=&\,(\HH_0\otimes Y)^{\otimes n},\end{align*}both with $n$ factors $Y$. Now let $U^-\to U$ be the cofibration given by the pushout-product axiom where $U^-$ is the union of objects like $U$ but with at least one of the $Y$'s replaced by an $X$, and similarly for $V^-\to V$. Then $U\cup_{U^-}V^-\rto V$ is also a cofibration by the pushout-product axiom. The map $A\to\KK(1,0)^{(n)}$ considered in (\ref{pout00cof}) is a map between pushouts as described by the following commutative cube

\begin{diagram}[small,w=0.8cm,silent,UO]&&V^-\otimes\HH(1,0)&\rTo&&&\KK(1,0)^{(n-1)}\\&\ruTo&\vLine&&&\ruTo&\dTo\\U^-\otimes\HH(1,0)&&&\rTo&\KK(1,0)^{(n-1)}&&\\\dTo&&\dTo&&\dTo&&\\&&V\otimes\HH(1,0)&\hLine&\VonH&\rTo&\KK(1,0)^{(n)}\\&\ruTo&&&&\ruTo&\\U\otimes\HH(1,0)&\rTo&&&A&&\end{diagram}in which front and back square are pushouts. It follows then from an easy diagram chase (cf. \cite[Lemma 6.9]{BM2}) that the induced square\begin{diagram}[small](U\cup_{U^-}V^-)\otimes\HH(1,0)&\rto&A\cup_{\KK(1,0)^{(n-1)}}\KK(1,0)^{(n-1)}\\\dTo&&\dTo\\V\otimes\HH(1,0)&\rTo&\KK(1,0)^{(n)}\end{diagram}is a pushout square. Therefore, since the left vertical map is a cofibration of right $\HH_1$-modules, the right vertical map (i.e. the map $A\to\KK(1,0)^{(n)}$) is as well, which is precisely what had to be shown.

Note that the same argument in fact shows that $\HH(1,0)\otimes_{\HH_1}\KK_1\to\KK(1,0)$ is a cofibration of right $\KK_1$-modules. Indeed, given a commutative square of right $\KK_1$-modules\begin{diagram}[small]\HH(1,0)\otimes_{\HH_1}\KK_1&\rTo^\chi&M\\\dTo&&\dTo\\\KK(1,0)&\rTo&N\end{diagram}write $\phi^{(0)}$ for the map $\KK(1,0)^{(0)}=\HH(1,0)\to M$ of $\HH_1$-modules corresponding to $\chi$ by adjunction, and proceed as above.\vspace{1ex}

\underline{Proof of (f)}. We used in the proof of (e) that $\KK_0$ carries a natural filtration $\cdots\to\KK_0^{(p-1)}\to\KK_0^{(p)}\to\cdots$ starting with $\KK_0^{(0)}=\HH_0$. We will first prove that $\partial\KK_0$ is similarly filtered by objects $(\partial\KK_0)^{(p)}$ which fit into a ladder\begin{diagram}[small]\partial\HH_0&=&(\partial\KK_0)^{(0)}&\rTo&(\partial\KK_0)^{(1)}&\rTo&(\partial\KK_0)^{(2)}&\rTo&\cdots\\\dTo&&\dTo&&\dTo&&\dTo&&\\\HH_0&=&\KK_0^{(0)}&\rTo&\KK_0^{(1)}&\rTo&\KK_0^{(2)}&\rTo&\cdots\end{diagram}in which all the maps are cofibrations, as are the comparison maps from the inscribed pushouts$$(\partial\KK_0)^{(p)}\cup_{(\partial\KK_0)^{(p-1)}}\KK_0^{(p-1)}\rto\KK_0^{(p)}.$$This will imply that the map $\partial\KK_0\to\KK_0$ in the colimit is a cofibration, and in fact that the canonical map $\HH_0\cup_{\partial\HH_0}\partial\KK_0\to\KK_0$ is a cofibration.

Similarly to the cofibrations $W^{(p)}_-\to W^{(p)}$ used to construct the filtration of $\KK_0$, cf. diagram (\ref{pout00filt00}), one can construct cofibrations using the pushout-product axiom,$$V^{(p)}\to W^{(p)}$$where$$V^{(p)}=\bigcup_{0\leq k\leq p}(\HH_0\otimes Y)^{\otimes k}\otimes\partial\HH_0\otimes(Y\otimes\HH_0)^{\otimes p-k}$$is the colimit over all objects like $W^{(p)}$ but with at least one occurence of $\HH_0$ replaced by $\partial\HH_0$. Let $V^{(p)}_-$ be the colimit of similar objects, where in addition at least one of the $Y$'s is replaced by an $X$. So the maps $X\to Y$ and $\partial\HH_0\to\HH_0$ induce a commutative square\begin{gather}\label{pout00partialK}\begin{diagram}[small]V^{(p)}_-&\rTo&V^{(p)}\\\dTo&&\dTo\\W^{(p)}_-&\rTo&W^{(p)}\end{diagram}\end{gather}in which (again by the pushout-product axiom) all maps are cofibrations, as is the comparison map\begin{gather}\label{pout00comparison}W^{(p)}_-\cup_{V^{(p)}_-}V^{(p)}\rto W^{(p)}\end{gather}from the inscribed pushout in (\ref{pout00partialK}) to the lower right corner. Now construct a sequence of cofibrations$$D^{(0)}\rto D^{(1)}\rto D^{(2)}\rto\cdots$$by setting $D^{(0)}=\partial\HH_0$, and constructing $D^{(p)}$ from $D^{(p-1)}$ as a pushout\begin{gather}\label{pout00D}\begin{diagram}[small]V^{(p)}_-&\rTo&D^{(p-1)}\\\dTo&&\dTo\\V^{(p)}&\rTo&D^{(p)}\end{diagram}\end{gather}for $p>0$. Let $D$ be the colimit $D=\varinjlim_pD^{(p)}$. Thus $D$ is filtered by the $D^{(p)}$ and one can now construct maps $D^{(p)}\to\KK_0^{(p)}$ starting with $\partial\HH_0\to\HH_0$ for $p=0$, and continuing from $p-1$ to $p$ by completing the cube below in which the left and right squares are the pushouts (\ref{pout00D}) and (\ref{pout00filt00}).

\begin{diagram}[small,silent,UO]&&D^{(p-1)}&\rTo&&&\KK_0^{(p-1)}\\&\ruTo&\vLine&&&\ruTo&\dTo\\V^{(p)}_-&&&\rTo&W^{(p)}_-&&\\\dTo&&\dTo&&\dTo&&\\&&D^{(p)}&\hDots&\VonH&\rDotsto&\KK_0^{(p)}\\&\ruTo&&&&\ruTo&\\V^{(p)}&\rTo&&&W^{(p)}&&\end{diagram}
Since the comparison map (\ref{pout00comparison}) of the front square is a cofibration, the comparison map $D^{(p)}\cup_{D^{(p-1)}}\KK_0^{(p-1)}\to\KK_0^{(p)}$ of the back square is a cofibration as well, cf.  \cite[Lemma 6.9]{BM2}. It thus suffices to show that $D\to\KK_0$ is isomorphic to $\partial\KK_0\to\KK_0$. By definition, $\partial\KK_0$ is the coequaliser of the following diagram\vspace{1ex}

\begin{gather}\label{pout00coeq}\begin{diagram}[small](\KK_0\otimes_{\HH_0}\HH(1,0))\otimes_{\HH_1}\KK_1\otimes_{\HH_1}(\HH(0,1)\otimes_{\HH_0}\KK_0)\\\dTo^\alpha\dTo_\beta\\(\KK_0\otimes_{\HH_0}\HH(1,0))\otimes_{\HH_1}(\HH(0,1)\otimes_{\HH_0}\KK_0)\\\dTo_\pi\\\partial\KK_0\end{diagram}\end{gather}

Now first of all, each constituent $\HH_0\otimes\cdots\otimes Y\otimes\partial\HH_0\otimes Y\otimes\cdots\otimes\HH_0$ of $V^{(p)}$ maps naturally to $\partial\KK_0$ since it can be rewritten as$$(\HH_0\otimes\cdots\otimes Y\otimes\HH(1,0))\otimes_{\HH_1}(\HH(0,1)\otimes Y\otimes\cdots\otimes\HH_0)$$which maps canonically to the middle object of the coequaliser (\ref{pout00coeq}). When composed with $\pi$ these together give a well-defined map $D^{(p)}\to\partial\KK_0$ for each $p$, and in the colimit we obtain a map $D\to\partial\KK^0$.

In the other direction, the filtrations of $\KK_0$, $\KK_1$ and $\KK_0$ by $\KK^{(p)}_0$, $\KK_1^{(r)}$ and $\KK_0^{(q)}$ respectively, induce a filtration by three degrees $(p,r,q)$ on the top object of (\ref{pout00coeq}), by two degrees $(p,q)$ on the middle object, and by one degree $n$ on the coequaliser $\partial\KK_0$. The maps $\alpha,\,\beta,\,\pi$ take the $(p,r,q)$-part to the the $(p+r,q)$-part respectively the $(p,r+q)$-part, and the $(p,q)$-part to the $(p+q)$-part. Now the filtration part $$(\KK_0^{(p)}\otimes_{\HH_0}\HH(1,0))\otimes_{\HH_1}(\HH(0,1)\otimes_{\HH_0}\KK_0^{(q)})$$ maps to $D^{(p+q)}$ in the obvious way, and this map factors through $\pi$ to give a natural map $(\partial\KK_0)^{(p+q)}\to D^{(p+q)}$. Together, these define a map $\partial\KK_0\to D$. It is now a straightforward diagram chase to check that the two maps thus constructed, $D\to\partial\KK_0$ and $\partial\KK_0\to D$, are mutually inverse.\qed\end{stit}

For the remaining proof of Lemma \ref{IAL} we need the following complement to Lemma \ref{cofibrant2}, where an object of a monoidal model category is called \emph{weakly contractible} if there is a zig-zag of weak equivalences relating it to the monoidal unit.

\begin{lma}\label{contractible}Let $R$ be a weakly contractible, well pointed monoid in $\Vv$, and let $M$ (resp. $N$) be a weakly contractible, cofibrant right (resp. left) $R$-module. Then the tensor product $M\otimes_RN$ is a weakly contractible, cofibrant object of $\Vv$.\end{lma}

\begin{proof}In virtue of Lemma \ref{cofibrant2} it remains to be shown that $M\otimes_RN$ is weakly contractible in $\Vv$. Observe that in any Quillen model category a zig-zag of weak equivalences between two cofibrant objects can be replaced by a zig-zag of weak equivalences between the same objects, which passes through cofibrant objects only. Therefore, the weakly contractible, cofibrant left $R$-module $N$ can be related to the monoidal unit $R$ by a zig-zag of weak equivalences passing through cofibrant left $R$-modules only. After application of the left Quillen functor $M\otimes_R-$, we thus get a zig-zag of weak equivalences in $\Vv$ relating $M\otimes_RN$ and $M$. Now $M$ itself is a weakly contractible, right $R$-module, hence there is a zig of weak equivalences between $M$ and $R$. Finally, the unit $I_\Vv\to R$ is a weak equivalence by assumption so that we get a zig-zag of weak equivalences between $M\otimes_RN$ and $I_\Vv$ as required.\end{proof}

\subsection{Proof of the Interval Amalgamation Lemma \ref{IAL}}\label{proofIAL}--\vspace{1ex}

Let $\partial_i:\{0,1\}\to\{0,1,2\}$ denote the order-preserving inclusion which omits $i$. The amalgamation $\HH*\KK$ of the $\Vv$-intervals $\HH$ and $\KK$ can then be defined by $$\HH*\KK=\partial_1^*(\partial_{2!}\KK\sqcup\partial_{0!}\HH)$$where the coproduct is taken in $\VCat_{\{0,1,2\}}$. We write $\LL=\partial_{2!}\KK\sqcup\partial_{0!}\HH$, hence $\HH*\KK=\partial_1^*\LL$ in $\VCat_{\{0,1\}}$. It remains to be shown that $\HH*\KK$ is weakly equivalent to $\Iso$. Since $\HH$ and $\KK$ are $\Vv$-intervals, there are weak equivalences $\HH\eqv\wIso$ and $\KK\eqv\wIso$ inducing a $\Vv$-functor $\HH*\KK\to\wIso*\wIso$. Note that $\wIso$ can be chosen to be a $\Vv$-interval itself; moreover, it is readily verified that $\Iso*\Iso\cong\Iso$. It is therefore sufficient to show that for \emph{any} $\Vv$-intervals $\HH$ and $\KK$, the amalgamation $\HH*\KK$ has weakly contractible hom-objects. This will follow from the Interval Cofibrancy Theorem \ref{mainthm} together with the following explicit description of the hom-objects $\LL(i,j)$, where as usual $\LL(i,i)$ is abbreviated to $\LL_i$:\vspace{1ex}

\begin{itemize}\item[-] $\LL_1=\HH_1*\KK_0\quad\text{(the coproduct of monoids)};$
\item[-] $\LL(0,1)=\LL_1\otimes_{\HH_1}\HH(0,1);$
\item[-] $\LL(1,0)=\HH(1,0)\otimes_{\HH_1}\LL_1;$
\item[-] $\LL(1,2)=\KK(0,1)\otimes_{\KK_0}\LL_1;$
\item[-] $\LL(2,1)=\LL_1\otimes_{\KK_0}\KK(1,0);$
\item[-] $\LL(0,2)=\LL(1,2)\otimes_{\LL_1}\LL(0,1);$
\item[-] $\LL(2,0)=\LL(1,0)\otimes_{\LL_1}\LL(2,1).$
\end{itemize}
The endomorphism-monoid $\LL_1$ is cofibrant as coproduct of two cofibrant monoids. Moreover, since $\HH_1$ and $\KK_0$ are weakly contractible monoids, their coproduct $\LL_1$ is a weakly contractible monoid as well. By construction, $\LL(0,1)$ is obtained by applying the left Quillen functor $\LL_1\otimes_{\HH_1}-$ to the cofibrant left $\HH_1$-module $\HH(0,1)$, hence $\LL(0,1)$ is a cofibrant left $\LL_1$-module. By hypothesis, $\HH(0,1)$ is weakly contractible in $\Vv$ and $\HH_1$ is a weakly contractible monoid. In particular, $\HH(0,1)$ is weakly contractible left $\HH_1$-module so that $\LL(0,1)$ is a weakly contractible left $\LL_1$-module.

Similarily, $\LL(2,1)$ is a weakly contractible, cofibrant left $\LL_1$-module, and $\LL(1,0)$ and $\LL(1,2)$ are weakly contractible, cofibrant right $\LL_1$-modules. Lemma \ref{contractible} thus implies that $\LL(0,2)$ and $\LL(2,0)$ are weakly contractible, cofibrant objects of $\Vv$.

The endomorphism-monoids of $\LL$ at $0$ and $2$ are given by the following pushouts, of $\HH_0$-bimodules and $\KK_1$-bimodules respectively (cf. the proof of Lemma \ref{bimoduleadjoint}):
\begin{diagram}[small]\HH(1,0)\otimes_{\HH_1}\HH(0,1)&\rTo&\HH_0&\quad\quad&\KK(0,1)\otimes_{\KK_1}\KK(1,0)&\rTo&\KK_1\\\dTo&&\dTo&&\dTo&&\dTo\\
\LL(1,0)\otimes_{\LL_1}\LL(0,1)&\rTo&\NWpbk\LL_0&\quad\quad&\LL(1,2)\otimes_{\LL_1}\LL(2,1)&\rTo&\NWpbk\LL_2\end{diagram}

Since (in virtue of Theorem \ref{mainthm}(ii) and Lemma \ref{contractible}) the left vertical maps of both squares are weak equivalences between weakly contractible, cofibrant objects of $\Vv$, and (in virtue of Theorem \ref{mainthm}(iii)) the upper horizontal maps are cofibrations in $\Vv$, the right vertical maps $\HH_0\to\LL_0$ and $\KK_1\to\LL_2$ are weak equivalences as well, and hence $\LL_0$ and $\LL_2$ are weakly contractible monoids as required.\qed

\section*{Acknowledgments}

We are grateful to Boris Chorny, Fernando Muro and Giovanni Caviglia for helpful discussions and useful comments. The first author benefited from support from the French National Agency for Research (ANR grants HODAG and HOGT).

\vspace{1ex}

\noindent{\small\sc Universit\'e de Nice, Lab. J.-A. Dieudonn\'e,
Parc Valrose, 06108 Nice, France.}\hspace{2em}\emph{E-mail:}
cberger$@$math.unice.fr\vspace{1ex}

\noindent{\small\sc Radboud Universiteit Nijmegen, Institute for Mathematics, Astrophysics, and Particle Physics, Heyendaalseweg 135, 6525 AJ Nijmegen, The Netherlands.}\hspace{2em}\emph{E-mail:}
i.moerdijk$@$math.ru.nl

\end{document}